\documentclass[10pt]{article}
\usepackage{authblk}
\usepackage{graphicx,bm,amssymb,amsmath,amsthm} 
\usepackage{xcolor}
\usepackage[colorlinks]{hyperref} 
\usepackage[capitalise, nameinlink]{cleveref}
\crefname{equation}{}{} 
\usepackage{array}
\usepackage{siunitx} 
\usepackage[caption=false]{subfig} 
\usepackage{bold-extra} 
\usepackage{bm} 
\usepackage{tabularx} 
\usepackage[export]{adjustbox} 
\usepackage{tikz}
\usepackage[most]{tcolorbox}
\usetikzlibrary{positioning}
\usepackage{comment}                
\usepackage[backend=biber,
style=numeric,
maxcitenames=5,
maxalphanames=5,
maxbibnames=5,
maxnames=5,
hyperref=true,
firstinits=true
]{biblatex}
\DeclareSIUnit\parsec{pc}

\newcommand{\bi}{\begin{itemize}}
\newcommand{\ei}{\end{itemize}}
\newcommand{\ben}{\begin{enumerate}}
\newcommand{\een}{\end{enumerate}}
\newcommand{\be}{\begin{equation}}
\newcommand{\ee}{\end{equation}}
\newcommand{\bea}{\begin{eqnarray}} 
\newcommand{\eea}{\end{eqnarray}}
\newcommand{\ba}{\begin{align}} 
\newcommand{\ea}{\end{align}}
\newcommand{\bse}{\begin{subequations}} 
\newcommand{\ese}{\end{subequations}}
\newcommand{\bc}{\begin{center}}
\newcommand{\ec}{\end{center}}
\newcommand{\bfi}{\begin{figure}}
\newcommand{\efi}{\end{figure}}

\newcommand{\bmp}[1]{\begin{minipage}{#1}}
\newcommand{\emp}{\end{minipage}}
 
\newcommand{\bp}{\begin{proof}}
\newcommand{\ep}{\end{proof}}
\newcommand{\ie}{{\it i.e.\ }}
\newcommand{\eg}{{\it e.g.\ }}

\newcommand{\etc}{{\it etc.\ }}

\newcommand{\mbf}[1]{{\mathbf #1}}

\newcommand{\C}{\mathbb{C}}

\newcommand{\R}{\mathbb{R}}

\renewcommand{\d}{\mathrm{d}}

\newcommand{\bigO}{{\mathcal O}}

\newcommand{\veps}{\varepsilon}

\DeclareMathOperator{\im}{Im}
\DeclareMathOperator{\re}{Re}

\newtheorem{thm}{Theorem}

\newtheorem{cor}[thm]{Corollary}
\newtheorem{pro}[thm]{Proposition}
\newtheorem{rmk}[thm]{Remark}
\newtheorem{defn}[thm]{Definition}
\newcommand{\om}{\omega}
\newcommand{\g}{\gamma}
\newcommand{\te}{\tilde\eta}
\crefalias{prop}{Proposition}

\begin{document}

\title{An adaptive spectral method for oscillatory second-order linear ODEs with frequency-independent cost }

\author[1,*]{Fruzsina J.\ Agocs}
\author[1,**]{Alex H.\ Barnett}
    \affil[1]{Center for Computational Mathematics, Flatiron Institute, 162 Fifth Avenue, New York, 10010 NY, USA}
    \affil[*]{fagocs@flatironinstitute.org}
    \affil[**]{abarnett@flatironinstitute.org}

\date{}
\maketitle

\begin{abstract}

We introduce an efficient numerical method for second order linear ODEs whose
solution may vary between highly oscillatory and slowly changing over the solution interval.
In oscillatory regions the solution is generated via a nonoscillatory phase
function that obeys the nonlinear Riccati equation.
We propose a
defect-correction iteration that gives an asymptotic series for such a phase
function; this is numerically approximated on a Chebyshev grid with a small number of nodes.
For analytic coefficients we prove that each iteration, up to a certain
maximum number,
reduces the residual by a factor of order of the local frequency. 
The algorithm adapts both the step size and the choice of method,
switching to a conventional spectral collocation method
away from oscillatory regions.
In numerical experiments we find that our proposal outperforms other state-of-the-art
oscillatory solvers, most significantly at low-to-intermediate
frequencies and at low tolerances, where it may use up to $10^6$
times fewer function evaluations.  Even in high frequency regimes, our
implementation is
on average $10$ times faster than other specialized solvers.
\end{abstract}

\section{Introduction}\label{introduction}

The efficient numerical solution of highly oscillatory ordinary differential
equations (ODEs) has long been a challenge.
A wide range of specialized methods exists for dealing with ODEs
with various structures;
see \cite{petzold1997,engquist2009} for thorough reviews.
In this work we handle a commonly occurring
form, the variable-coefficient, linear, second
order ODE with
smooth coefficients.
This includes simple harmonic oscillators with time-dependent coefficients,
which arise in many computational physics problems.
One application in which an extremely fast numerical solution is needed
is Bayesian parameter estimation in cosmology.
Specifically, each spatial Fourier mode of the density perturbations
obeys such an ODE in time,
with solutions that may switch repeatedly between smooth and
highly oscillatory characters \cite{agocs2020dense}.
Of order $10^3$ modes are needed, and this must be repeated for
up to $10^6$ trial parameter choices; thus the ODE solver
may be called a billion times, and an efficient frequency-independent
solver is crucial.
Besides inflationary cosmology \cite{Hergt2022,martin2003,winitzki2005}, this class of ODE appears in
Hamiltonian dynamics \cite{Pritula2018,fiore2022}, particle accelerators and plasma physics
\cite{courant1958,davidson2001,hazeltine2003,lewis1968}, electric circuits
\cite{likharev2022}, satellite systems \cite{saxena2020}, acoustics and gravity waves \cite{filippi1998,einaudi1970}, and quantum mechanics \cite{griffiths2018,adhikari1988,arnold2011wkb,cea1982}.
In many of these fields our proposed method could supply a robust and efficient
numerical replacement for
uncontrolled asymptotic approximations.

The ODE under study has the form
\be
u''(t) + 2\g(t) u'(t) + \om^2(t)u(t) = 0, \quad t \in (t_0,t_1) \subset \R,
\label{ode}
\ee
where $\g$ is a given smooth (by which we mean $C^\infty$) damping function,
and $\om$ is a given smooth, but possibly very large, real-valued
local frequency function\footnote{If coefficients are merely piecewise smooth, a solution may simply be patched together at the breakpoints, thus we do not discuss this situation further.}.
We solve the initial value problem (IVP)
\begin{align} u(t_0) &= u_0, \label{ic0} \\
    u'(t_0) &= u'_0, \label{ic1}
\end{align}
although we note that the methods that we present could easily be applied to
two-point boundary value problems.
The restriction to real $\om$ is not crucial,
but avoids complications with exponentially large dynamic ranges
(e.g., deeply evanescent regions)
that we leave for future study.

In regions where $\om \gg 1$,
the solution $u$ is oscillatory:
loosely speaking,
it has local approximate form $a e^{i\om_0t} + b e^{-i\om_0t}$, where $\om_0$
is a local value of $\om(t)$.
Conventional ODE integrators then require discretization with several
grid points per local period $2\pi/\om_0$, forcing the overall
cost to be $\bigO(\om_0)$, which can be prohibitively slow.
This has led to the development of solvers
that exploit new representations or asymptotic expansions of the solution, allowing many periods between discretization nodes.
We now review such prior work,
while noting that, to our knowledge, our proposal is the first
tackle the problem adaptively with spectral accuracy
when the solution has both highly oscillatory and smooth regimes.

The most common asymptotic expansion
is the Wentzel--Kramers--Brillouin (WKB), a
perturbation method applied to linear differential equations that contain a
small parameter
in the highest derivative
\cite{benderorszag,logan}.
This approximation arose in quantum mechanics where it is still widely
applied for approximate analytic solutions of the Schr\"{o}dinger equation.
Writing $\om(t) = \om_0\Omega(t)$,
where $\Omega(t)$ is of unit size, and the asymptotic parameter is
$\om_0 \gg 1$,    \cref{ode} becomes
\be\label{odewkb}
u'' + 2\gamma(t) u' + \om_0^2 \Omega(t)^2 u = 0.
\ee
Substituting $u(t) = e^{\om_0 z(t)}$ and defining $z'(t) = x(t)$
gives the Riccati equation (a nonlinear 1st-order ODE) for $x$, 
\be\label{riccwkb}
x' + 2\gamma(t) x + \om_0 x^2 + \om_0 \Omega(t)^2 = 0.
\ee
Making a regular perturbation expansion in $1/\om_0$,
\be\label{wkbexpansion}
x(t) = \sum_{j=0}^{N} \om_0^{-j} x_j(t),
\ee
and matching powers of $\om_0$ in \cref{odewkb} gives the recursion relation
for the functions $x_j(t)$
(where for simplicity we give only the $\gamma \equiv 0$ case),
\begin{equation}\label{wkbrecur}
  x_0 = \pm i\Omega, \quad x_j = -\frac{1}{2x_0}\left( x'_{j-1} + \sum_{k = 1}^{j-1}x_k x_{j-k} \right), \quad \text{for } j = 1, 2, \ldots.
\end{equation}
Although originally an analytic tool,
WKB has recently been exploited in numerical solvers.
These algorithms achieve
$\bigO(1)$ (\ie $\om_0$-independent) runtime in regions where the asymptotic
series is a sufficiently good approximation.
WKB forms the basis of the method described in
\cite{agocs2020efficient,agocs2020dense}, and associated open-source software
package \texttt{oscode} \cite{agocs2020joss}. This uses the expansion
\cref{wkbexpansion} up to and including $x_3$ within a time-step,
or, if this is not accurate enough, switches to a 4,5th
order Runge--Kutta pair.
In \cite{arnold2011wkb,korner2022wkb} the related
``WKB marching method'' uses the same
expansion truncated after $x_2$ to transform \cref{odewkb} into a smoother
problem in the highly oscillatory regions,
otherwise again switching to a 4,5th order Runge--Kutta method.
In terms of the stepsize $h$, neither method is formally high-order.
\texttt{oscode} has
$\bigO(h)$ global error, although
its prefactor is small when $\om_0$ is large.
The
WKB-marching method's oscillatory solver improves on this and has $\bigO(h^2)$ global error, resulting in better performance at lower tolerances.

Yet, as a consequence of low $h$ order and fixed asymptotic order,
a small requested tolerance in either method forces
$h \lesssim 1/\om_0$, removing the efficiency advantage over conventional timesteppers.
They are therefore only efficient if no more than about 6 digits of accuracy
are required, and can suffer at low-to-intermediate frequencies where more asymptotic terms would be required.
The present method, on the other hand, is both spectrally accurate and 
adaptive in the number of terms in its asymptotic expansion.

The ``quasilinearization method''
(QLM) \cite{bellman1970} instead successively approximates
a solution of the Riccati equation \cref{riccwkb}
by a functional Newton iteration.
The resulting differential equation recurrence (here stated for $\gamma\equiv 0$),
\be
x_0 = \pm i\om, \quad
x_j' - x_{j-1}^2 + 2x_j x_{j-1} + \om^2 = 0, \quad \text{for } j = 1, 2, \ldots
\label{qlmode}
\ee
turns out to yield a solution equivalent to a WKB expansion with $2^j$ terms.
This recurrence is either solved symbolically \cite{mandelzweig2004,krivec2006}
or numerically \cite{krivec2008,krivec2014}. The former
requires increasing amounts of algebra with iteration number, so can
become cumbersome. Integrating each ODE in \cref{qlmode} numerically does not prove to be a better strategy since each such ODE is itself oscillatory.

Recently Bremer \cite{bremer2018}, following \cite{heitman2015,bremer2016},
has proposed an efficient
method for solving \cref{ode}, in the special case $\g \equiv 0$,
over a domain in which the solution is highly oscillatory.
The method aims to find a global nonoscillatory phase function $\alpha(t)$
satisfying Kummer's equation,
\be
\frac{3}{4} \left(\frac{\alpha''}{\alpha'}\right)^2 - \frac{1}{2}\frac{\alpha'''}{\alpha'} - (\alpha')^2 +
\om(t)^2 = 0,
\label{kummer}
\ee
a nonlinear and generally oscillatory ODE, and will therefore be referred to as the Kummer's phase function method hereafter.
One then reconstructs solutions
via $u(t) = \sin(\alpha(t) + c)/\sqrt{\alpha'(t)}$.
The generated phase functions are nonoscillatory in the sense that they can be
accurately approximated on intervals of fixed size
using polynomials of degree independent of $\om_0$.
The results of \cite{heitman2015,bremer2016}
prove the existence of such nonoscillatory phase functions when
$\Omega(t)$ is
nonoscillatory in the sense that $\log \Omega$ is smooth and its Fourier transform rapidly decaying.
Finding initial conditions on the phase function leading to
a nonoscillatory solution is, however, challenging.
Bremer proposed a ``windowing'' method to numerically find such a global phase function on the interval $(t_0,t_1)$,
first using a $C^\infty$ partition of unity near $t_1$ to smoothly ``blend'' $\Omega$ to
a constant function, integrating backwards from this constant region
down to $t = t_0$, and finally integrating $t \in (t_0,t_1)$ forward using the original $\Omega$.
This exploits two facts: a nonoscillatory phase function for
the constant case $\Omega(t) \equiv 1$ is simple to write down, and
a smooth window allows the nonlinear Kummer oscillator to make a smooth 
(``adiabatic'')
transition which excites only an exponentially small oscillation amplitude.
The parameters of this partition of unity require problem-dependent
adjustment.
Nevertheless Bremer's is the most efficient
and accurate prior method known to the authors\footnote{During the preparation
of this manuscript, a new and more general phase function-based method able to
tackle zeros in $\om$ has been proposed by Bremer
\cite{bremer2022turningp}. No comparison to this novel solver is made in this
work due to its current lack of a publicly available implementation.}.

Here we present an approach using an asymptotic expansion for the derivative of
a nonoscillatory Riccati phase function, meaning that we construct an expansion
for $x(t) = (\log u(t))'$, the phase function being $z(t) = \log u(t) =
\int^t x(\sigma)\d \sigma$.
An asymptotic sequence of functions approximating $x(t)$
is constructed
via a defect
correction scheme \cite{bohmer1984},
where each iteration involves simply taking a derivative
(performed numerically on a spectral grid), rather than solving
an ODE as in QLM.
We adaptively terminate this (non-convergent) iteration before it becomes unstable.
Our algorithm is capable of handling a non-zero damping term $\g$,
allowing for the direct solution of a wider range of second order linear ODEs
than the WKB-marching or the Kummer's phase function method.
Our iteration does not introduce oscillatory contributions,
hence finds a nonoscillatory approximate phase function without
reference to initial conditions, in contrast to the
Kummer's phase function method. 
Another crucial difference from the latter is that $x(t)$ is solved for
locally on one interval with an adaptively-chosen length, rather than globally.
Solutions to \cref{ode}--\cref{ic1} are then constructed by time-stepping over the
integration range, repeatedly matching Cauchy data at interval endpoints.
Our iteration only involves one previous term (a 2-term recurrence),
so is simpler than the traditional WKB expansion
\cref{wkbrecur} requiring all past terms.

Phase function methods, as with WKB or QLM methods,
is naturally only applicable in the high frequency regime.
Thus, a significant part of this work is to generalize
ideas from \cite{agocs2020efficient} to
create a fully adaptive algorithm that can switch 
between oscillatory and nonoscillatory methods,
and choose the time-step $h$ to achieve a user-requested
tolerance efficiently.
Our nonoscillatory method is a standard adaptive spectral collocation scheme.

This paper is structured as follows. \cref{methods} describes the various
numerical methods used in the proposed algorithm including the asymptotic
expansion for oscillatory domains, the spectral method based on Chebyshev
nodes, adaptive control of stepsize, and switching between the two alternative
schemes.
\cref{errorana} presents rigorous bounds
(given analytic coefficients)
on the Riccati residual
under our defect correction scheme,
including a bound $\bigO(e^{-c\om_0})$ given sufficient terms.
We show results from numerical experiments in
\cref{numresults}, which involves a comparison with state-of-the-art ODE solvers
mentioned above. \cref{conclusions} concludes with a summary and
suggestions for future work.

\section{Methods \label{methods}}

\subsection{Defect correction of the Riccati phase function}
  \label{phasefun}

In this section we explain how to solve for
a nonoscillatory Riccati phase function on an interval,
which will correspond to a single time-step of the overall ODE solver.
This interval has to be sufficiently small so that changes in the coefficients $\om(t)$,
$\g(t)$ are small, yet it has to contain a large number of oscillations of $u(t)$, so that $\om(t)$ is sufficiently large.
Writing $u = e^z$ where $x(t) = z'(t)$ is the derivative of the phase function $z(t)$,
any nonvanishing $u(t)$ satisfies \cref{ode} if and only if $x(t)$ satisfies the nonlinear Riccati equation
\be
x' + x^2 + 2\g(t)x + \om(t)^2 = 0.
\label{ricc}
\ee
(For convenience we now work without rescaling by $\om_0$ as in \cref{introduction}.)
A generic solution to \cref{ricc} is oscillatory
on the local timescale $1/\om(t)$.
This is illustrated by the case of constant $\om\equiv \om_0$
and $\g\equiv 0$,
which has only the family of analytic solutions
$x(t) = \om_0 \tan(\alpha - \om_0t)$ parameterized by $\alpha\in\C$,
where $\re \alpha$ is interpreted as a phase shift and $\im \alpha$
controls the amplitude. Their oscillation frequency is $2\om_0$.
However, taking the limit $\im \alpha \to \pm \infty$ produces
the only two nonoscillatory (constant) solutions $x(t) = \pm i\om_0$.
For general analytic frequency functions $\om(t)$ there
also exist nonoscillatory solutions.
These take the approximate form
$x_{\pm}(t) \approx \pm i\om(t)$, leading to a conjugate pair of basis
functions $u_{\pm}(t) := \exp(\int^t x_\pm(\sigma) \d \sigma)$.
The rigorous existence of nonoscillatory solutions
follows from the corresponding result
for Kummer's equation \cref{kummer} proven in
\cite{heitman2015,bremer2016}.
The argument that existence of a real-valued Kummer's phase function implies
existence of a Riccati phase function goes as follows:\footnote{We thank Jim Bremer for explaining this argument to us.}
if $\alpha(t)$ satisfies \cref{kummer},
then defining $\beta' = -\alpha''/2\alpha'$ allows \cref{kummer}
to be written
$\beta'' - (\alpha')^2 + (\beta')^2 + \om(t)^2 = 0$, and these last two equations
are the real and imaginary parts of \cref{ricc}
with $x = i\alpha' + \beta'$.

Given any trial solution $x$ to \cref{ricc}, we define its \emph{residual}
as the left-hand-side of \cref{ricc}, namely
\be
\label{R}
R[x](t) := x' + x^2 + 2\g(t)x + \om(t)^2.
\ee
We propose the following iteration
to generate a sequence of functions $x_0, x_1, \dots, x_k$
on a given interval,
using the above positive (say) approximate form as a starting point:
\begin{align}
x_0(t) &= +i\om(t), \label{init} \\
x_{j+1}(t) &= x_j(t) - \frac{R[x_j](t)}{2 \left( x_j(t) + \g(t) \right)}, \qquad j=0,1,\dots,k-1. \label{iter}
\end{align}
This is an \emph{approximate} Newton iteration for \cref{ricc} (note that an
exact Newton iteration has been used numerically \cite{krivec2008,krivec2014}
and as an analysis tool \cite{heitman2015}). To show this, perturbing $x_j$ by
a function $\delta_j$ we get
\begin{align}
R[x_j + \delta_j] &= x_j' + \delta_j' + x_j^2 + 2 x_j \delta_j + \delta_j^2 
+ 2\g x_j + 2\g\delta_j + \om(t)^2 \nonumber \\
&= R[x_j] + \delta_j' + 2x_j\delta_j + 2\g\delta_j + \bigO(\delta_j^2), \nonumber
\end{align}
then linearization yields a linear first order ODE for $\delta_j$, 
$\delta_j' + 2x_j(t) \delta_j + 2\g \delta_j  = -R[x_j](t)$.
However, this ODE is itself generally oscillatory,
with an unknown initial condition needed to achieve a nonoscillatory
solution.
Yet if $\delta_j$ is nonoscillatory, the first term $\delta_j'$ is
a factor $x_j = \bigO(\om)$ smaller than the second term, thus
one might hope that by dropping it a useful reduction in residual might
still result.
This leads to the algebraic formula $2\left(x_j(t) + \g(t)\right) \delta_j(t) = -R[x_j](t)$
which is solved pointwise for each $t$. Updating via
$x_{j+1} = x_j + \delta_j$ explains our proposed \cref{iter}.
This is a type of defect correction scheme \cite{bohmer1984}.
If $\om$ is a smooth function, then
the result is nonoscillatory by construction, without explicit reference
to an initial condition.

The early iterates of \cref{init}--\cref{iter} illustrate
the residual reduction (we give the case $\g(t) \equiv 0$, for
simplicity, but the results hold for the general case):
\begin{align}
x_0 &=  i\om, &&R[x_0] = i\om' = \bigO(\om), \nonumber \\
x_1 &= i\om - \frac{\om'}{2\om}, &&R[x_1] = -\frac{\om''}{2\om} +
    \frac{3\om'^2}{4\om^2} = \bigO(1), \label{x1R1} \\
x_2 &= i\om - \frac{\om'}{2\om} + \delta_1, \text{ where }
    \delta_1 = \frac{\frac{\om''}{2\om} - \frac{3\om'^2}{4\om^2}}{2i\om - \frac{\om'}{\om}}, \quad
    &&R[x_2] = \delta_1' + \delta_1^2 = \bigO(\om^{-1}). \nonumber
\end{align}
Here our somewhat loose asymptotic notation $\bigO(\om)$, \etc 
has meaning over some interval where $\om$ does not change much,
\ie that its upper and lower bounds are of the same order.
Also, since $\om$ is assumed smooth, we have that $\om'$, $\om''$, \etc are
of the same order as $\om$.
If we were to include $\g$, we would assume $\g = \bigO(1)$
and that its derivatives are of the same order as $\g$ itself due to smoothness.
This suggests that each iteration reduces the residual by a factor $\bigO(\om)$
but brings in one extra derivative of $\om$ (and $\g$).
To verify this pattern for arbitrary $j$, notice that
$$
R[x_{j+1}] = R[x_j] + 2\left(x_j + \g \right)\delta_j  + \delta_j^2 + \delta_j' = \delta_j^2 + \delta_j',
$$
where the first two terms canceled by design, due to \cref{iter}.
Substituting $\delta_j = -R[x_j]/2(x_j + \g)$ and using
the quotient rule then proves the following useful update
for the residual function.
\begin{pro}\label{PRiter}  
  Let $x_j$, $\om$ and $\gamma$ be smooth functions on an interval
  on which $x_j+\gamma$ is nonvanishing.
  Let $x_{j+1}$ be given by the iteration \cref{iter}.
  Then the associated residual function \cref{R} obeys the iteration
  \be
  \label{Riter}
    R[x_{j+1}] = \frac{1}{2(x_j + \g)}\left( \frac{(x_j + \g)'}{x_j + \g} R[x_j] - R[x_j]' \right) 
    + \left(\frac{R[x_j]}{2(x_j + \g)}\right)^2. 
  \ee
\end{pro}                  
Thus, if one can show that $|x_j+\gamma|$ is uniformly bounded from
below by an $\bigO(\om)$ constant,
and that the final quadratic term in \cref{Riter} is small,
then one might hope that the
residual is an asymptotic series
shrinking like $R[x_j] = \bigO(\om^{1-j})$ for $j=0,1,\dots$.
However, it will turn out that this series, and the iteration \cref{iter} for $x_j$, are not convergent:
for a typical fixed functions $\om(t)$ and $\g(t)$,
the increasing order of derivatives eventually lead to divergence.
Nevertheless, the iteration
is useful numerically, since when $\om$ varies only a small amount
over an interval,
the residual decays geometrically for the first few iterations with a rate that improves with larger $\om$.
We will show this rigorously for the case of analytic $\om$ and $\g$
in \cref{TR}.

\subsection{Discretization of Riccati defect correction}
\label{discr}

We now turn to the numerical discretization of
the iteration \cref{init}--\cref{iter} on an interval
$[t_i,t_i+h]$ corresponding to the $i$th time-step of the solver.
For now, assume that $h$ is sufficiently small that the functions
$\om$, $\g$, $x_j$, $R_j$, and their first few derivatives are accurately
represented by function values on a single $(n+1)$-node Chebyshev grid.
We will explain our choice of $h$ in \cref{hosc}.
The remaining tasks are numerical differentiation of the function
$x_j$ needed to evaluate $R_j$ in \cref{R}, and
numerical quadrature (since $u(t) = \exp \int^t x(\sigma) \mathrm{d}\sigma$).

Recall that on the standard interval $[-1, 1]$ the $n+1$ Chebyshev nodes are
$\tilde\tau_l = \cos\left( l\pi/n\right)$, $l = 0, 1, \dots, n$.
In our experiments we fixed $n=16$, which we found gave a good compromise
between very high order and speed of numerical linear algebra.
The standard $(n+1) \times (n+1)$ differentiation matrix $\tilde{D}_n$
is filled according to \cite[Ch.~6]{tref}.
Rescaling to the time-step interval $[t_i, t_i+h]$
gives nodes
\be\label{scaledt}
\tau_l = t_i + \frac{h}{2}[1+\cos (l\pi/n)], \quad l = 0, 1, \ldots, n,
\ee
and the rescaled matrix $D_n = (2/h) \tilde D_n$.
Then given a vector $\mbf{x}\in\C^{n+1}$ with components
${x_l := \{x(\tau_l)\}}$ values of a smooth function $x$ at the nodes,
$D_n\mbf{x}$ is a vector of approximate derivatives at these nodes.
We then fill the vectors of values $\om_l=\om(\tau_j)$
and $\g_l=\gamma(\tau_j)$.
The discretization of the iteration \cref{init}--\cref{iter}
gives the sequence $\mbf{x}^{(0)}, \mbf{x}^{(1)},\dots,\mbf{x}^{(k)}$ whose
components $l=0,\dots, n$ obey
$$
x^{(0)}_l = i\om_l, \quad
x^{(j+1)}_l = x^{(j)}_l - \frac{R^{(j)}_l}{2(x^{(j)}_l + \g_l)}, \text{ where }
R^{(j)}_l = \sum_{m=0}^n (D_n)_{lm}x^{(j)}_m  + [x^{(j)}_l]^2 + 2\g_lx^{(j)}_l + \om_l^2 .
$$
We postpone choice of the stopping point $k$ to \cref{steptype}.

Finally, we must discretize the antiderivative operation giving
$z(t)$ on the interval such that $z' = x$, to then recover a solution $u = e^z$.
If dense output is required, \ie the solver is
to perform interpolation and compute the numerical solution between
internal time nodes, then this is done by constructing the integration matrix
$Q_n$ mapping function values at $n$ Chebyshev nodes to values of the
antiderivative of the interpolating polynomial at those points (with the last value
being zero). We use \texttt{chebfun}'s \cite{chebfunguide,chebfunquad}
implementation which constructs $Q_n$ as $Q_n = T_n B_n T_n^{-1}$, where $T_n$
and $T_n^{-1}$ convert Chebyshev interpolation polynomial coefficients to
node values and vice versa, and $B_n$ is the appropriate Vandermonde matrix. $T_n$
($T_{n-1}$) is equivalent to the (inverse) discrete cosine transform
\cite{gentlemanclenshawI,gentlemanclenshawII}. The matrix $Q_n$ is computed once and for all at
$\bigO(n\log n)$ cost, then at each timestep $Q_n \mbf{x}$ is performed,
followed by a rescaling by $h/2$.
If, on the other hand, the solution is required only at the end of the interval, $t = t_1$, then
$z$ need not be evaluated at the Chebyshev nodes and it suffices to compute the
definite integral $\int_{x_i}^{x_i+h}x\mathrm{d}t$ at each timestep. This is
done by Clenshaw--Curtis quadrature, following \cite[Chapter 12]{tref}.

\subsection{Numerical demonstration of Riccati defect correction}
\label{demo}

At this point it is useful to illustrate the residual reduction
with an example.
In the left panel of \cref{residual-plot}, we apply
the above numerically-discretized iteration to the
equation
\begin{equation}\label{bursteq}
  u'' + \om^2(t) u = 0, \quad  \om(t)=\frac{\sqrt{m^2 - 1}}{1+t^2},
\end{equation}
over a single timestep interval $t \in [0, 0.5]$ covered by a $n=16$ Chebyshev grid.
This choice of $\om(t)$ is analytic in the neighborhood of the interval,
having poles at $\pm i$.
For various fixed choices of $m$, which controls
the maximum frequency according to $\om_{\text{max}} := \sqrt{m^2-1}$,
we plot the maximum magnitude residual vector $\max_{0\le l \le n} |R_l^{(j)}|$
at the nodes, as a function of iteration number $j$.
We observe that the residual
indeed drops geometrically for the first few iterations,
with a rate well described by $(\om_{\text{max}})^{-j}$.
Thus, at very large $\om_{\text{max}} \approx 10^4$ the
convergence is extremely rapid, with machine precision reached at $j=5$.
However, for a small $\om_{\text{max}}=10$, the iteration
suddenly starts diverging beyond $6$
iterations, so that only about $5$ digits of accuracy can be achieved
by optimal stopping.
While divergence is to be expected since the series generated
by the iteration is asymptotic rather than convergent,
one would expect (as in \cref{errorana})
the divergent iteration to be $\bigO(\om_{\text{max}})$.
Yet, curiously,
for other intermediate $\om_{\text{max}}$ values
a sudden deterioration of rate (``kink'') occurs
also at the \emph{same} $j=6$.
This suggests that the rate change comes from the same mechanism,
one unrelated to asymptotic divergence.

We now discuss the cause of this kink phenomenon.
With asymptotics ruled out,
one might well then suspect discretization error:
the action of the
matrix $D_n$ does not give the true derivative at the nodes,
and this compounds with each iteration.
However, the right panel of \cref{residual-plot} provides strong
evidence for a third explanation: the amplification of rounding errors.
The case of constant $\om(t)$ provides a convenient diagnostic:
since its Riccati residual is zero in exact arithmetic
(both in the continuous and discrete cases),
only rounding error remains.
This error model is shown by the new dotted curves in the plot:
they clearly show that the kink occurs when rounding error dominates.
Here we picked the value of the constant to be the minimum of $\om(t)$ in
\cref{bursteq} over the interval $t \in [0, 0.5]$,
since, for fixed $n$ and $h$, smaller values of constant
$\om$ result in larger discretization errors after $j > 1$ iterations.
We suspect that the constancy with $\om_{\text{max}}$
of the kink location $j$ is due to the difference in rounding error
slope vs residual reduction slope (on the semi-log plot) being
a constant controlled only by the largest eigenvalue of $D_n$.
We find that the position of the kink is determined by a number of parameters,
including: the ratio of the distance to the nearest pole of $\om$ and the
stepsize $h$; the number of nodes $n$; and the magnitude of $\om$. Detailed
investigations
are deferred to future work.
\begin{rmk}\label{fewer}  
  In \cref{pracres} we will show that, in exact arithmetic,
  even though the number of iterations $k$
  required to reach minimal residual grows as $\bigO(\om)$, in
  practice a larger $\om$ requires \emph{fewer} iterations to reach a fixed
  accuracy threshold $\varepsilon$.
  Indeed, solving $\varepsilon\approx \om^{-k}$ gives a heuristic
  $k_\varepsilon \approx \log(1/\varepsilon)/\log \om$.
\end{rmk}

\begin{figure*}[tb]  
    \flushleft
    \subfloat{\includegraphics{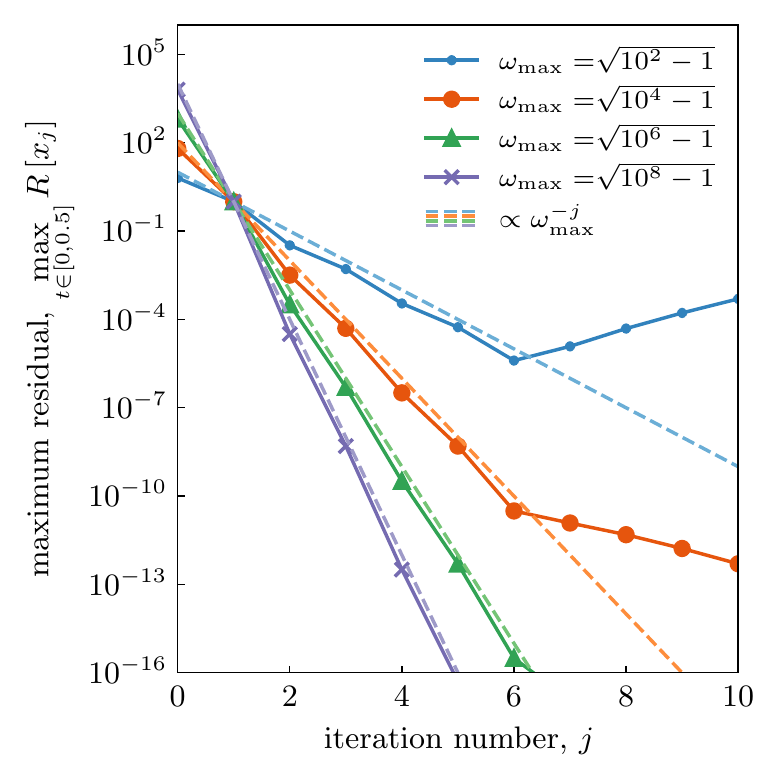}}
\subfloat{\includegraphics{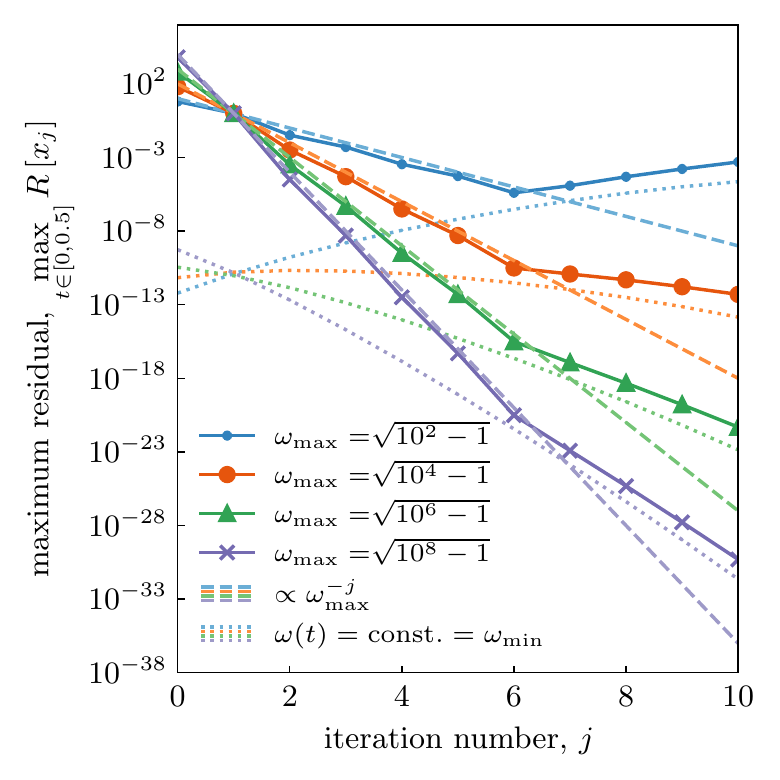}}
    \caption{\label{residual-plot}
      Demonstrations of Riccati residual defect correction.
      In detail, the left plot shows the maximum magnitude of the residual
      of the Riccati ODE \cref{R} over the
      $n=16$ Chebyshev grid, as a function of iteration count,
      for the example in \cref{demo}.
      For small $\om_{\text{max}}$ an initial geometric decay eventually becomes
      unstable growth for large $j$. The dotted lines show $(\om_{\text{max}})^{-j}$.
      The right plot contains the same information as the left plus
      a set of dotted lines showing our model of roundoff error amplification in
      the discretized Riccati iteration, computed by setting $\om$ to be
      constant (see text).
The ``kink'' at $j=6$ is explained by this error
      becoming dominant.
}
\end{figure*}

\subsection{Matching Riccati timestep intervals}

So far we have explained how to numerically solve for 
one approximate nonoscillatory Riccati phase function over one timestep
$[t_i,t_i+h]$,
in the sense of having uniformly small residual on its grid nodes.
However, this cannot match general initial conditions
$u(t_i)$ and $u'(t_i)$, as needed for a second order ODE timestepper.
Yet, the asymptotic iteration \cref{iter} can be
initialized with $x_0(t) = +i\om(t)$ (as done in \cref{init}), or with $x_0(t)
= -i\om(t)$. The two resulting Riccati solutions, $x_{\pm}(t)$,
give two linearly independent approximate solutions of the desired ODE
\cref{ode}, namely
\be
u_{\pm}(t) = e^{\int_{t_i}^t x_{\pm}(\sigma)\d\sigma}.
\label{upm}
\ee
Since $\om$ and $\gamma$ are real, $u_{+}$ and $u_{-}$ are complex conjugates.
They may be linearly combined to match any set of complex-valued
initial conditions,
\begin{align}
    u(t) &= A_{+}u_{+}(t) + A_{-}u_{-}(t), \qquad t \in [t_i, t_i + h],\\
    u'(t) &= A_{+}u'_{+}(t) + A_{-}u'_{-}(t), \qquad t \in [t_i, t_i + h],
\end{align}
with coefficients $A_\pm$ satisfying the $2\times 2$ linear system
\be
\begin{bmatrix}
    1 & 1 \\
    x_{+}(t_i) & x_{-}(t_i)
\end{bmatrix}
\begin{bmatrix}
    A_{+} \\
    A_{-}
\end{bmatrix}
= 
\begin{bmatrix}
    u(t_i) \\
    u'(t_i)
\end{bmatrix}.
\ee
Here we used that $u_{\pm}(t_i) = 1$ and $u'_{\pm}(t_i) = x_{\pm}(t_i)$ by construction of \cref{upm}.

\subsection{Direct spectral solution on nonoscillatory intervals \label{chebysteps}}

If the solution is nonoscillatory, or the asymptotic expansion fails to
converge before the residual reaches a user-specified tolerance $\varepsilon$, the solver switches to a standard spectral collocation method on a Chebyshev grid.
This achieves an arbitrarily high order timestep in a simple manner.
(We note that in other applications a \emph{global} spectral grid is
advantageous \cite{tref}; here, since adaptive switching and efficiency are
needed,
integration is best performed over many timestep intervals with a fixed high
order grid on each.)
Using the
Chebyshev nodes $\{\tau_l\}_{l=0}^n$ and differentiation matrix $D_n$
from \cref{discr},
Eq.~\cref{ode} may then be discretized over the interval $t \in [t_i, t_i+h]$ by
the linear system
\begin{align}\label{discreteode}
  F_n \mbf{u} := \left(D_n^2 + 2\text{diag}(\{\g_l\})D_n + \text{diag}(\{\om_l^2\}) \right)\mbf{u} = \mbf{0},
\end{align}
where $\mbf{u}:=\{u(\tau_l)\}_{l=0}^n$ is the node value vector, and
diag indicates the diagonal matrix constructed from a vector.
One must also enforce initial conditions via
\be\label{discreteic}
u_n
= u(t_i), \qquad \left(D_n\mbf{u}\right)_n = u'(t_i).
\ee
Stacking these conditions as rows gives the overdetermined $(n+3) \times (n+1)$ linear system
\be\label{discreteodeic}
\renewcommand*{\arraystretch}{1.25}
\begin{bmatrix}
    F_n \\
    (D_n)_{n,:} \\
    (\mathbb{I}_n)_{n,:}
\end{bmatrix}
\mbf{u}
= 
\begin{bmatrix}
\mbf{0}_n \\
    u'(t_i) \\
    u(t_i)
\end{bmatrix},
\ee
which is solved in the least-squares sense via SVD with $\bigO(n^3)$ cost.
$u(t)$ at the interval's upper end can be read off as the first
element (since nodes $\tau_l$ are ordered backwards) of the
solution vector $\mbf{u}$, and its derivative as the first element of
$D_n\mbf{u}$.
These then serve as the initial conditions for the subsequent time-step.

We estimate the error induced over this timestep $[t_i,t_i+h]$
by comparing to the solution at the interval endpoint using a doubled value
of $n$.
If the relative difference falls below the user-specified
threshold $\varepsilon$, the step with doubled $n$ is accepted ($t_i$ is advanced by $h$).
Otherwise the timestep is rejected ($t_i$ is unchanged), and the algorithm
repeats with $h$ halved.
Computation time spent in failed steps is lost, making
it imperative that the initial stepsize estimate $h$ is chosen carefully. We
describe the procedure for choosing $h$ for both spectral and asymptotic steps
in detail in the next section.

\subsection{Adaptive selection of interval size and type}
\label{adap}

In this section we explain the adaptive stepsize and method-switching
algorithm. The left side of \cref{balls-flowchart} shows
a summary flowchart.

\subsubsection{Initial stepsize estimates}
\label{hosc}

At each timestep start $t_i$, a decision needs to be made about whether to
use the Riccati defect correction or the Chebyshev spectral method.
Both types of steps could be attempted
and the decision based upon the their estimated error (as done in \cite{agocs2020efficient}).
Instead, we find it more efficient to use local properties of $\om$ at $t_i$
to get initial stepsize estimates.
Recalling \cref{x1R1}, the first correction term is $\om'/\om$, giving the
approximate timescale
\be\label{hoscini}
h_{\text{osc}} = \frac{\om(t_i)}{\om'(t_i)}.
\ee
\cref{TR} will also support this by showing that
the ratio of bounds on $\om'$ and $\om$ appears in the rate of
convergence of successive Riccati residuals.
In practice to approximate $\om'(t_i)$, we use Chebyshev differentiation
from the grid of the previous timestep, and from the grid of an initial
stepsize at $t=t_0$, supplied by the user.

A similar estimate for the spectral method should depend on two things: how
oscillatory the solution is, and on what timescales the coefficients in \cref{ode} change, \ie their smoothness. Using the former as our criterion,
a measure of the rate of
oscillations is simply the frequency, giving the timescale
\be\label{hsloini}
h_{\text{slo}} = \frac{1}{\om(t_i)}.
\ee

These initial stepsize estimates are too crude to use directly, since
$\om$, $\om'$, or $\g$ may change significantly over such an interval.
Nonetheless they provide a useful starting point for refinement.

\subsubsection{Refining the stepsize estimates}

The discretization used in the Riccati defect correction method in \cref{discr} is certainly no more accurate than
the error in which the coefficient functions $\om$ and $\g$ are represented
on the Chebyshev grid.
Therefore we estimate their Chebyshev interpolation error as follows.
Let $\mbf{f}\in\R^{n+1}$ indicate the vector of nodes values
$\{f(\tau_l)\}$ for a function $f$ which will be either $\om$ or $\g$.
Let $f^\star(t)$ be the Lagrange polynomial interpolant through the
nodes $\{\tau_l\}_{l=0}^n$.
Let $\tau^\star_l := t_i + \tfrac{h}{2}[1 + \cos(l\pi/n + \pi/(2n)]$, $l=0,\dots,n-1$ be the set of $n$ points that lie halfway in between the $n+1$ Chebyshev nodes on the unit half-circle.
Then define the error estimate
\be
    \Delta_n[f] := \max_{l = 0, 1, \ldots, n} \biggl| \frac{f(\tau^{\star}_l) -
    f^{\star}(\tau^{\star}_l)}{f(\tau^{\star}_l)} \biggr|.
\ee
Since $\{f^{\star}(t^{\star}_l)\} = L_n \mbf{f}$ for some interpolation matrix
$L_n$ that is independent of the interval choice, in practice
$L_n$ is computed only once, \eg by solving a Vandermonde system in a backwards
stable fashion.
(Even though such a system is exponentially ill-conditioned,
the interpolation itself is very accurate up to $n \lesssim 40$
for Chebyshev nodes; see \cite[Appendix A]{helsing_close}.)

Letting $\Delta$ denote the larger of $\Delta_n[\om]$ and $\Delta_n[\g]$, we accept or update the stepsize as follows,
\be
h \leftarrow \begin{cases}
        h &\text{if } \Delta \leq \varepsilon_h, \\
        \min \left( 0.7h, 0.9 h \left( \frac{\varepsilon_h}{\Delta} \right)^{\frac{1}{n+1}} \right) &\text{otherwise},
    \end{cases}
    \ee
where $\veps_h$ is a parameter quantifying relative tolerance (distinct from $\veps$). In the above, the $1/(n+1)$ exponent is justified by noting that the
error in ($n+1$)-point Chebyshev interpolation is proportional to $h^{n+2}$, therefore if
$\Delta$ exceeds $\varepsilon_h$, using an exponent slightly larger than $1/(n+2)$ takes
the local error back down to a value smaller than $\varepsilon_h$, which is further
ensured by multiplying by the safety factor 0.9. We decrease the step to $0.7h$
if it yields a smaller stepsize to ensure quicker convergence if $\Delta$ is
only slightly larger than $\varepsilon_h$. The factors 0.7 and 0.9 were set based on
empirical tests.

The local error of Chebyshev spectral steps depends on both the stepsize and
the number of nodes. Since within a step we iterate over both of these until
convergence, we only need to ensure that over our proposed stepsize the
timescale over which $u$ changes ($\approx 1/\om$) does not increase too much. Starting from $h =
h_{\text{slo}}$, we refine the stepsize proposal according to
\begin{align}
    h \leftarrow& \begin{cases}
        h/2 &\text{if } \min\limits_{j = 0, 1, \ldots, n}\frac{1}{\om(t^{\star}_j)} < \sigma h \\
        h &\text{otherwise},
    \end{cases}
\end{align}
with $\sigma = 0.8$ chosen experimentally.

\subsubsection{Choosing the step type}
\label{steptype}

With a refined $h_{\text{osc}}$ and $h_{\text{slo}}$ in hand, we can make an informed decision about which type of step to take.
\begin{pro}\label{steptypechoose}
    Let $h_{\text{osc}}$ and $h_{\text{slo}}$ be the refined stepsize proposals
    for a step to be taken from $t = t_i$ using the asymptotic expansion and the Chebyshev
    spectral method, respectively. The algorithm chooses to attempt an
    asymptotic step of size $h_{\text{osc}}$ if and only if
$$ h_{\text{osc}} > 5h_{\text{slo}} \quad \text{and} \quad \omega(t_i) h_{\text{osc}} > 2\pi, $$
    otherwise it will attempt a spectral step of size $h_{\text{slo}}$.
\end{pro}
The reason behind requiring the proposed stepsize for an asymptotic step not to
simply be larger than that of a Chebyshev spectral step is that in regions
where the two stepsizes compete, the asymptotic method rarely converges before
the residual reaches $\varepsilon$, either because $\om$ is not large enough, or
either of $\g$ and $\om'$ is too large. The prefactor of $5$ was set based on numerical experiments. 

If an asymptotic step was decided to be attempted, the algorithm will start
iterating over $x_j$ according to \cref{iter}, while
monitoring the maximum residual over the nodes.
If $\max_{l=0,\dots,n} |R[x_j](\tau_l)| < \varepsilon$ or $\max_{l=0,\dots,n}
|R[x_j](\tau_l)| \geq \max_{l=0,\dots,n} | R[x_{j-1}](\tau_l)|$, the iteration is stopped. In the former
case, the proposed solution $u(t_i+h)$ is accepted and the independent variable
is incremented, $t := t_i + h$. Otherwise, the step is retried with size
$h = h_{\text{slo}}$ using the spectral method.

If either of the conditions in \cref{steptypechoose} were not met, a spectral
step is attempted with $h = h_{\text{slo}}$. The local error of the step is
then estimated and the number of nodes, as well as the stepsize, is
adapted until a local error value of at most $\varepsilon$ is reached, as described in
\cref{chebysteps}.

We summarize the entire adaptive switching algorithm
in the left panel of \cref{balls-flowchart}.

\section{Error analysis for the Riccati iteration\label{errorana}}

Here we prove our main rigorous result:
the defect correction of \cref{phasefun}
in the case of analytic coefficients induces
temporary geometric decrease of the Riccati residual by a factor $\bigO(\om)$ per iteration.
Indefinite geometric decrease is not in general possible because of
factorial growth due to repeated differentiation;
the iteration is asymptotic in $\om\gg 1$, but not convergent.
We abbreviate the residual function by $R_j = R_j(t) := R[x_j](t)$.
Our main tools in proving the following will be \cref{PRiter} (the iteration for $R_j$), induction, and complex analysis.

\begin{thm}\label{TR}   
  Fix $t\in\R$, and let the frequency function $\om$ and damping function $\g$ be analytic
  in the closed ball $B_\rho(t) := \{z\in\C : |z-t| \le \rho\}$,
  for some $\rho>0$, with the bounds in this ball
    \begin{alignat}{4}
        \eta_1 \leq \; &|\om(z)| && \leq \eta_2,          \label{ommag} \\
         &|\om'(z)| &&\leq \eta_3 \leq \frac{\eta_1^2}{17}, \label{omder} \\
         &|\g(z)| &&\leq \eta_4 \leq \frac{\eta_1^2}{34\eta_2}. \label{gammaupper}
  \end{alignat}
Let the nonnegative integer $k$ be small enough that
  \be
    r := \frac{1}{2\te_1 \rho} \left(1 + \frac{\te_2}{\te_1}\right) k + \frac{\te_3}{4\te_1^2} \leq \frac{3}{4},
  \label{r}
  \ee
  where
    \begin{align}
    \te_1 &= \eta_1 - \eta_4 - \frac{17 \te_3}{4 \eta_1},  \label{eta1} \\ 
    \te_2 &= \eta_2 + \eta_4 + \frac{17 \te_3}{4 \eta_1},   \label{eta2} \\
    \te_3 &= \eta_3 + 2\eta_2\eta_4. \label{eta3}
\end{align}
  Then after any number $j\le k$ of iterations of \cref{init}--\cref{iter},
  the function $x_j$ has Riccati residual \cref{R} bounded at the point $t$ by
  \be
  |R_j(t)| \le \te_3 r^j. \label{Rjbnd}
  \ee
\end{thm}  

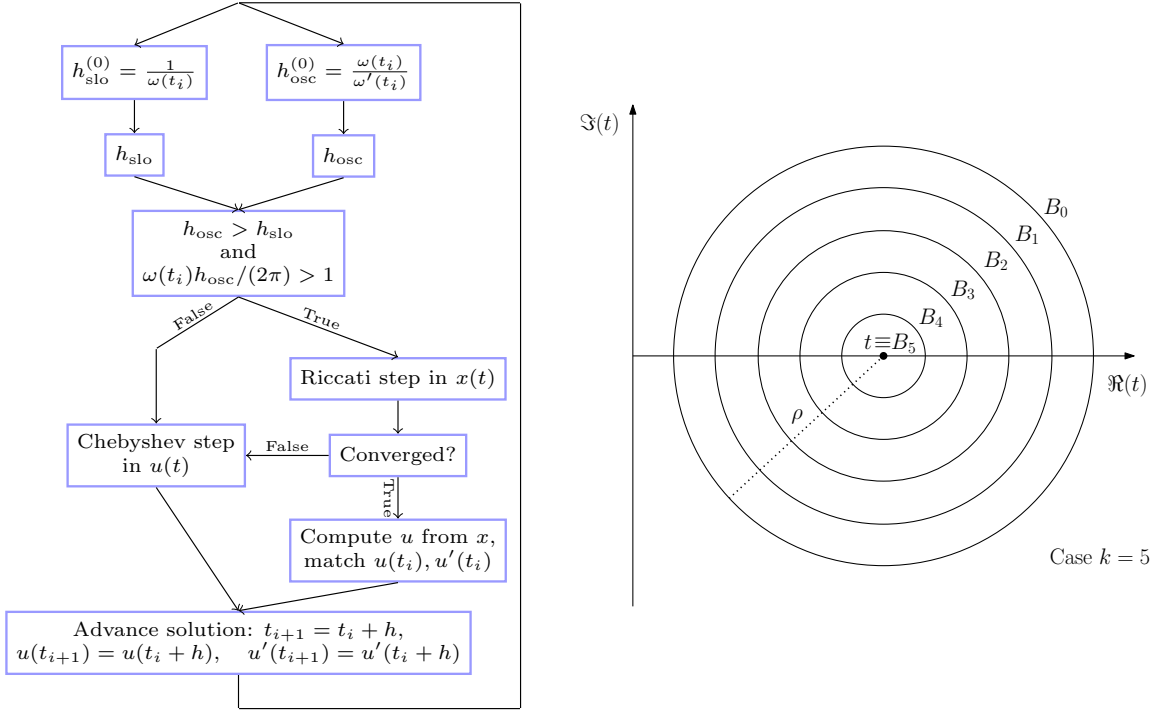
\begin{figure*}[tb]
    \centering
    \subfloat{\begin{adjustbox}{width=0.45\textwidth, valign=c}\begin{tikzpicture}[ roundnode/.style={circle, draw=black!60, fill=green!0, very thick, minimum size=7mm}, squarednode/.style={rectangle, draw=blue!40, fill=red!0, thick, minimum size=5mm}]

\tikzstyle{every node}=[font=\scriptsize]

\node[squarednode]      (hsloini)   {$h_{\text{slo}}^{(0)} = \frac{1}{\omega(t_i)}$};
\node[squarednode]      (hoscini)  [right=0.7cm of hsloini] {$h_{\text{osc}}^{(0)} = \frac{\omega(t_i)}{\omega'(t_i)}$};
\node[squarednode]      (hslo)  [below=0.4cm of hsloini] {$h_{\text{slo}}$};   
\node[squarednode]      (hosc)  [below=0.4cm of hoscini] {$h_{\text{osc}}$};    
\node[squarednode, align=center]      (switch) [below right=0.4cm and -0.4cm of hslo] {$h_{\text{osc}} > h_{\text{slo}}$ \\ and \\$\omega(t_i) h_{\text{osc}}/(2\pi) > 1$};
\node[squarednode]      (ricc)  [below right=3.1cm and -1.6cm of hoscini] {Riccati step in $x(t)$};  
\node[squarednode]      (conv)  [below=0.4cm of ricc] {Converged?}; 
\node[squarednode, align=center]      (cheb)  [left=1cm of conv] {Chebyshev step \\ in $u(t)$}; 
\node[squarednode, align=center]      (ic)  [below=0.5cm of conv] {Compute $u$ from $x$, \\ match $u(t_i), u'(t_i)$};     
\node[squarednode, align=center]      (accept)  [below =3.8cm of switch] {Advance solution: $t_{i+1} = t_{i}+h$, \\ $u(t_{i+1}) = u(t_i + h), \quad u'(t_{i+1}) = u'(t_i + h)$ };    
\node[inner sep=0, minimum size=0, above=0.9cm of cheb] (k) {}; \node[inner sep=0, minimum size=0, below=0.4cm  of accept] (l) {}; \node[inner sep=0, minimum size=0, right=3.4cm of l] (m) {}; \node[inner sep=0, minimum size=0, above=2.5cm of switch] (o) {}; \node[inner sep=0, minimum size=0, right=3.4cm of o] (n) {}; 

\draw[->] (hsloini.south) -- (hslo.north);
\draw[->] (hoscini.south) -- (hosc.north);
\draw[->] (hosc.south) -- (switch.north);
\draw[->] (hslo.south) -- (switch.north);
\draw[->] (switch.south) -- (ricc.north) node[midway, above=-0.5ex, sloped] {\tiny True};    
\draw[-] (switch.south) -- (k) node[midway, above=-0.5ex, sloped] {\tiny False};  
\draw[->] (k) -- (cheb.north) ;
\draw[->] (conv.south) -- (ic.north) node[midway, below=-0.5ex, sloped] {\tiny True};    
\draw[->] (conv.west) -- (cheb.east) node[midway, above=-0.5ex] {\tiny False};
\draw[->] (ic.south) -- (accept.north); 
\draw[->] (cheb.south) -- (accept.north);    
\draw[->] (ricc.south) -- (conv.north);
\draw[-] (accept.south) -- (l) ;
\draw[-] (l) -- (m) ;    
\draw[-] (m) -- (n) ;    
\draw[-] (n) -- (o) ; 
\draw[->] (o) -- (hsloini.north) ;
\draw[->] (o) -- (hoscini.north) ;    

\end{tikzpicture}
 \end{adjustbox}}
    \hfill
    \subfloat{\includegraphics[width=0.5\textwidth, valign=c]{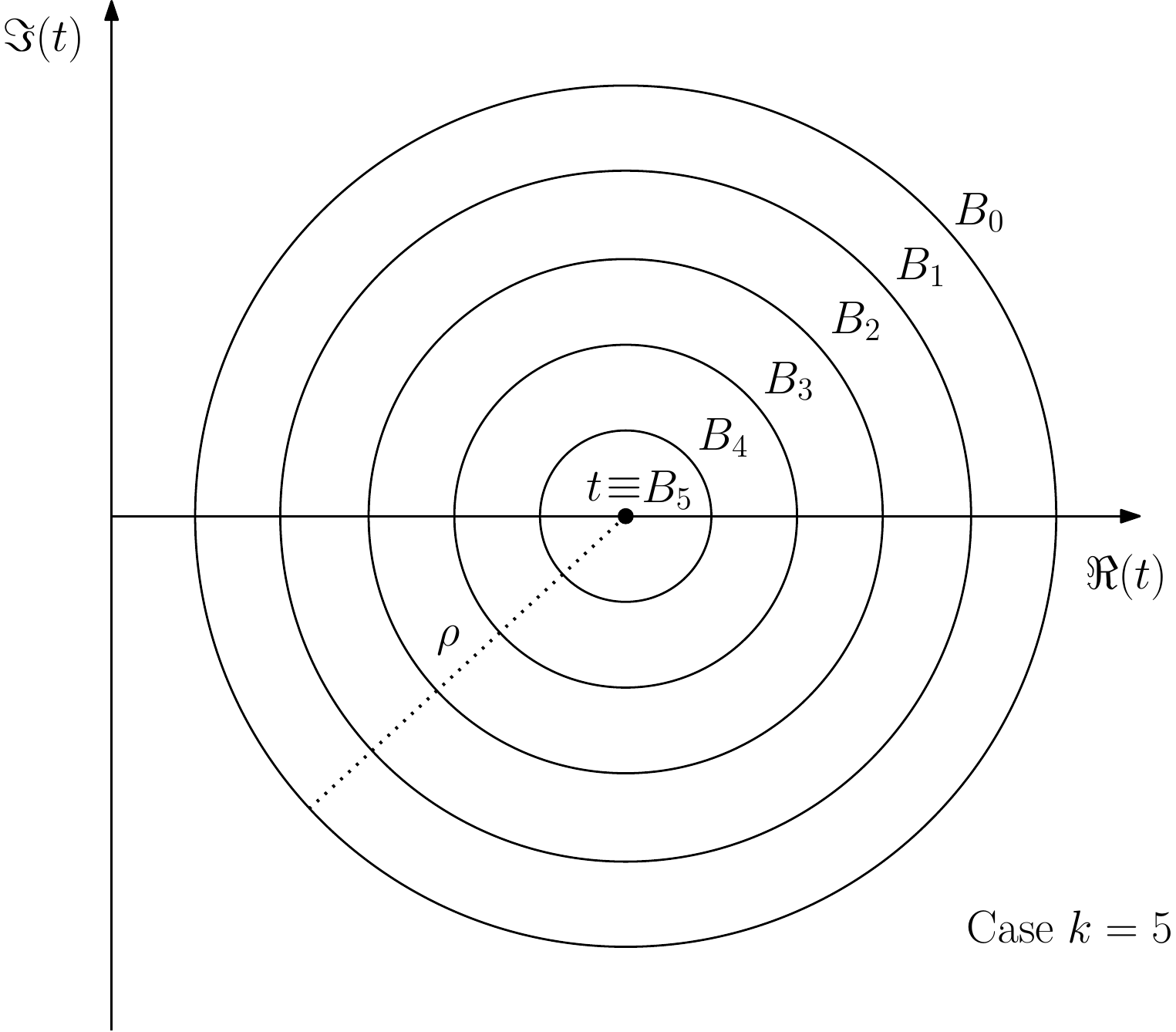}}
    \caption{\label{balls-flowchart} Left: Flowchart showing choice of stepsize $h$ and step type (Riccati vs standard Chebyshev) for our proposed algorithm.
      $t_i$ is the start of the current timestep.
    Right: Sketch of a nested equispaced set of
      balls centered on the target time $t$, used in the proof of \cref{TR}.}
\end{figure*}

This shows, for $\om'$ and $\g$ sufficiently small,
temporary geometric convergence up to $k$ iterations,
but at a rate $r$ that deteriorates with $k$.
For any $k$ to satisfy \cref{r}, $\om$ must have
a sufficiently large lower bound $\eta_1$, \ie $t$
must be in a sufficiently oscillatory region for the original ODE.
The condition \cref{omder} on $\om'$ is similar to
that in WKB: the \emph{relative $\om$ change per period} should be less than some small constant.

\begin{rmk}\label{slight}
  By construction, $[\te_1,\te_2]$ is only a limited
  expansion of the interval $[\eta_1,\eta_2]$ containing the range of $|\om|$ in the ball.
  In particular, by the inequalities \cref{ommag}--\cref{gammaupper} and
  definitions \cref{eta1}--\cref{eta3},
$\te_1 \ge 8\eta_1/17$ and $\te_2 \le \eta_2 + 9\eta_1/17$.
  Applying this we see that the last term in the definition of $r$ in \cref{r}
  is no more than $17/128 \leq 0.14$, thus can at worst cause only modest
  deterioration in the rate.
\end{rmk}

\begin{rmk}
  In the limit when $\om$ tends to be constant in the ball ($\eta_3$ small) and damping is small ($\eta_4$ small),
  then the geometric rate $r$ tends to $k/\om \rho$, where we can interpret $\om\rho$ as $2\pi$
  times the number of oscillation periods across the ball radius.
  Then, for example, with a radius of only 5 periods, $r$ satisfies \cref{r}
  for $k\le 25$.
\end{rmk}

\begin{proof}
  Define the concentric nested set of closed balls $B_j := B_{\rho_j}(t)$,
  with radii $\rho_j := (1-j/k)\rho$, $j=0,1,\dots,k$. Note that
  $B_0$ is the original ball in the statement of the theorem, while
    $B_k = \{t\}$ is the single point of interest; see the
    right panel of \cref{balls-flowchart}.
  For any function $f$ analytic in $B_j$ we abbreviate its $\infty$-norm by
  $\|f\|_j := \max_{z \in B_j}|f(z)|$.
  We will bound $f'$ in $B_{j+1}$ in terms of $\|f\|_j$ by
  applying Cauchy's theorem for derivatives
  \cite{steinshakarchi},
  \be
  f'(z) = \frac{1}{2\pi i} \int_{|\zeta-t|=\rho_j} \frac{f(\zeta)\, d\zeta}{(\zeta-z)^2}, \qquad z \in B_{j+1}.
  \label{cauder}
  \ee
  Bounding the integrand, then using the cosine rule we get
  $$
  |f'(z)| \le 
  \frac{\|f\|_j}{2\pi} \int_{|\zeta-t| = \rho_j} \frac{|d\zeta|}{|\zeta-z|^2}
  =
  \frac{\|f\|_j}{2\pi} \int_0^{2\pi} \frac{\rho_j\, d\theta}{|z|^2 + \rho_j^2 - 2|z|\rho_j \cos \theta}.
  $$
  We now use $\int_{0}^{2\pi} d\theta /(a + b \cos \theta) = 2\pi/\sqrt{a^2-b^2}$
  for $b<a$ \cite[Eq.~3.613.1]{GR8}, with
  $a = |z|^2 + \rho_j^2$ and $b = -2|z|\rho_j$, so that
  $\sqrt{a^2-b^2} = \rho_j^2-|z|^2$.
  Noting that the case $|z| = \rho_{j+1}$ bounds the others,
  and using $\rho_j=(1-j/k)\rho$, we compute, for any $0\le j < k$,
  \be
  \|f'\|_{j+1} \le
  \frac{\|f\|_j}{2\pi} \frac{2\pi\rho_j}{\rho_j^2-\rho_{j+1}^2}
  =
  \frac{\|f\|_j}{2 \rho} \frac{k(k-j)}{2k-2j-1}
  =
  \frac{\|f\|_j}{2 \rho} k \left[ \frac{1}{2} + \frac{1}{2(2k-2j-1)}\right]
  \le
  \frac{k}{\rho}\|f\|_j.
\label{derbnd}
  \ee
  Note that this bound is a factor $\bigO(k)$ better than naively
  lower bounding the denominator in \cref{cauder}.
  
We now use induction in iteration number $j$.
We take as the induction hypothesis that $x_\ell$ (and thus $R_\ell$)
  is analytic in $B_j$, for all $0\le \ell \le j$, with
  \begin{align}
      \te_1 \leq |x_\ell + \g| &\leq \te_2
  \qquad \mbox{ in } B_j, \quad \mbox{ for all } \ell = 0,1,\dots,j,
  \label{hypx} \\
      |R_\ell| &\leq \te_3 r^\ell \quad \mbox{ in } B_j, \quad \mbox{ for all } \ell = 0,1,\dots,j,
  \label{hypR}
  \end{align}
  where $r$ is defined by \cref{r}.
  Assuming for now the hypothesis for $j$, we apply simple bounds to the
  residual iteration \cref{Riter},
  lower-bounding denominator magnitudes and upper-bounding numerators
  via \cref{hypx},
  and applying \cref{derbnd} to the two derivative terms, to get
  $$
  \|R_{j+1}\|_{j+1}
  \leq
  \frac{1}{2\te_1}\left(
  \frac{1}{\te_1} \frac{k}{\rho}\te_2 + \frac{k}{\rho}
  \right)\|R_j\|_j
  +\biggl(\frac{\|R_j\|_j}{2\te_1}\biggr)^2
  \leq
  r \cdot \|R_j\|_j
  \leq
  \te_3 r^{j+1},
  $$
where in the middle step we used the crude bound $\|R_j\| \le \te_3$
following from \cref{hypR} and that $r\le1$, and then the definition of $r$ in \cref{r}.
The lower cases $\ell\le j$ follow trivially from the hypothesis since the balls
are nested.
Thus \cref{hypR} is proven for $j+1$.

It remains to verify that \cref{hypx} also holds for $j+1$.
By the functional iteration \cref{init}--\cref{iter},
$$
x_{j+1}(z) + \g(z) = i\om(z) + \g(z) - \sum_{\ell=0}^{j} \frac{R_\ell(z)}{2 \left(x_\ell(z) + \g(z)\right)},
\qquad z\in B_{j+1}.
$$
Using the hypothesis, the sum is pointwise bounded in magnitude
in $B_{j+1}$ by
$$
\left|\sum_{\ell=0}^{j} \frac{R_\ell}{2\left( x_\ell + \g \right)} \right|
\leq \frac{\te_3}{2\te_1} \cdot \sum_{\ell=0}^j r^{\ell}
\leq \frac{\te_3}{2\te_1} \cdot 4,
$$
where the upper bound in \cref{r} was used to bound the geometric series.
Using \cref{eta1} for $\te_1$ and \cref{eta3} for $\te_3$ we write this upper bound as
$$
\left|\sum_{\ell=0}^{j} \frac{R_\ell}{2\left( x_\ell + \g \right)} \right|
\leq
2\frac{\eta_3 + 2\eta_2\eta_4}{\eta_1 - \eta_4 - \frac{17\eta_3}{4\eta_1} - \frac{17\eta_2\eta_4}{2\eta_1}},
$$
then extend it by lower-bounding the denominator via \cref{omder} and the crude bound ${\eta_4 \leq \eta_1/34}$ from \cref{gammaupper}. This gives
$$
\left|\sum_{\ell=0}^{j} \frac{R_\ell}{2\left( x_\ell + \g \right)} \right| 
\leq
\frac{17\eta_3}{4\eta_1} + \frac{17\eta_2\eta_4}{2\eta_1},
$$
which, when combined with $|\om| \leq \eta_2$ and $|\g| \leq \eta_4$ yields
$$\|x_{j+1} + \g\|_{j+1} \le \eta_2 + \eta_4 + \frac{17\eta_3}{4\eta_1} + \frac{17\eta_2\eta_4}{2\eta_1} = \te_2,$$
verifying the upper bound \cref{hypx} for $j+1$.
Instead combining with $|\om| \ge \eta_1$ gives
$$ |x_{j+1}| \ge \eta_1 - \eta_4 - \frac{17\eta_3}{4\eta_1} - \frac{17\eta_2\eta_4}{2\eta_1} = \te_1 $$
in $B_{j+1}$, which verifies the lower bound \cref{hypx}.
Again, the hypothesis is inherited for $\ell\le j$ by the nesting of the balls.

Finally, the base case $j=0$ for induction
satisfies \cref{hypx}--\cref{hypR}
by the conditions of the theorem,
since $x_0(t) = i\om(t)$, noting $[\eta_1,\eta_2] \subset [\te_1,\te_2]$,
while $R_0(t) = i\left(\om'(t) + 2\g\om \right)$ is bounded in magnitude by $\te_3$.
\end{proof}

In the above theorem, the choice of $3/4$ in \cref{r} was merely a convenient one, chosen so that the bounds $[\te_1,\te_2]$ on $|x_j+\g_j|$
involving the geometric sum were
not much wider than the bounds $[\eta_1,\eta_2]$ on $|\om|$.
Note that in the case of $\gamma\equiv 0$ it is possible to
simplify the statement and proof of the theorem somewhat,
and improve some of the constants.

\subsection{Practical aspects of residual reduction \label{pracres}}

The fact that, as $k$ grows,
the rate $r$ bound in \cref{TR} deteriorates, we believe to be
an inevitable consequence of
the series generated by the iteration being asymptotic in $1/\om$ but not in general convergent.
However, given a function $\om(t)$ uniformly large enough
in a ball,
by stopping at a roughly optimal $k$ (an idea
called ``superasymptotics'' \cite{berrysuper,boydsuper})
one can achieve
exponential convergence with respect to the size of the frequency $\om$,
as follows.

\begin{cor}[Superasymptotic approximation]\label{super}
  Suppose $\om$ and $\g$ satisfy the conditions \cref{ommag}--\cref{gammaupper}
  about a point $t\in\R$,
  and let $\alpha := (1+\te_2/\te_1)/2\te_1\rho$ be the
  first term in the rate \cref{r} arising from these bounds.
  Then there is a nonnegative integer $k$ such that
  $$
  |R_k(t)|  \le  e \te_3 e^{-1/5\alpha}.
  $$
\end{cor}
\begin{proof}
  Set $k$ to be the integer in the interval $[1/5\alpha-1, \, 1/5\alpha)$.
    Then since the final term in \cref{r} is bounded by $17/128 \leq 0.14$
    by \cref{slight}, and the
    first term is $k\alpha \le 1/5$, we get $r\le e^{-1}$, which
    is also $\le 3/4$ so that \cref{Rjbnd} holds.
    Choosing $j=k$, and inserting the lower bound on $k$,
    $|R_k(t)| \le \te_3 r^k \le \te_3 e^{-(1/5\alpha-1)}$, which
    finishes the proof.
  \end{proof}
Recalling that $\alpha = \bigO(1/\om)$, this
shows that an $\om$-dependent
number of iterations can give, in exact arithmetic, an
exponentially convergent pointwise residual,
$$
|R(t)| = \bigO(e^{-c\om})
$$
for some constant $c>0$.
In the limit where
$\om$ tends to constant in the ball of radius $\rho$ about $t$, then
$\alpha$ tends to $1/\om\rho$, so the above
constant is roughly $c\approx \rho/3$.
Equivalently, roughly one decimal digit is achieved per
$1.1$ periods of oscillation across the ball radius.

However, this optimal number of iterations grows as $k = \bigO(\om)$,
so for large $\om$ is impractical, and unnecessary.
In practice we use the adaptive residual stopping criterion for $k$
described in \cref{steptype}.
By returning to \cref{TR} and dropping the small second term
in \cref{r} we get, in the limit of constant $\om$ in the ball,
$$
|R_k(t)|  \approx  \te_3 \left( \frac{k}{\rho\om}\right)^k.
$$
This shows \textit{nearly} geometric convergence in $k$,
sufficient to reach machine accuracy efficiently with a few iterations
when $\om$ is large.
For large $\rho\om$ one may approximately invert this to predict the iteration number $k$ sufficient for $|R_k| \approx \varepsilon$, to get
$ k_\varepsilon \approx \frac{\log \varepsilon^{-1}}{\log \rho\om}$,
a refinement of \cref{fewer}.

Finally we show that the stopping criterion for $k$ based on the absolute
residual $R[x_k]$ of the Riccati equation is relevant because this is equal
to the \emph{relative residual} of the original ODE.
\begin{pro}[Relationship of residuals]\label{residualu}
    Let $R[x_k](t)$ be the residual of the Riccati equation as defined in
    \cref{R}.
Then if
    $\tilde{R}[u_k](t)$ is the associated residual of the original ODE
    \cref{ode}, defined as
    \be\label{Rode}
    \tilde{R}[u_k](t) := u''_k + 2\gamma(t)u'_k + \omega^2(t) u_k,
\ee
    for $u_k(t) = e^{\int_0^t x_k(\sigma)\mathrm{d}\sigma}$, then
\be
    \frac{\tilde{R}[u_k](t)}{u_k(t)} = R[x_k](t).
    \ee
\end{pro}
\begin{proof}
    This follows straightforwardly from substitution of the above form of $u_k$ into \cref{Rode}.
\end{proof}

\section{Numerical results \label{numresults}}

Our numerical experiments will test the integration of 2nd-order linear
ODEs that include regions with highly-oscillatory solutions.
In such regions the condition number of the problem itself becomes
large, and this limits the achievable accuracy of \emph{any}
numerical method that works with finite-precision input data.
Recall \cite[Ch.~6]{GCbook}
that the definition of relative condition number for evaluation
of a fixed function $f(t)$ is $\kappa(t) := |tf'(t)/f(t)|$.
As a toy example,
for the oscillatory function $f_{\omega_0}(t) = e^{i\om_0t}$, 
the condition number of evaluating $f_{\omega_0}(1)$ is $\kappa(1) = \om_0$,
the total phase accrued
over the interval $[0,1]$.
However, since $\om_0$ is also an input to this problem,
the relative condition number
\emph{with respect to variation in} $\om_0$ is also relevant.
The latter is
$|\om_0/f_{\om_0}(1)| \cdot |\partial f_{\om_0}(1)/\partial \om_0| = \om_0$,
in this case the same as the condition number with respect to $t$.
Thus, if $\om_0=10^{7}$, so that $\kappa=10^7$ by either definition,
then using double-precision arithmetic with
$\varepsilon_{\mathrm{mach}} \approx 10^{-16}$
no algorithm can achieve an error less than $\bigO(\kappa\cdot \varepsilon_{\mathrm{mach}}) \approx 10^{-9}$,
the error for a backward-stable algorithm.

For more general oscillatory solutions to the IVP \cref{ode}--\cref{ic1},
the two above types of condition number generally differ,
and the larger controls the minimum achievable error.
For example, scaling $\om(t)=\om_0 \Omega(t)$ as in \cref{introduction},
then taking the zero-order approximation
$u(t) = e^{\pm i \om_0 \int_{t_0}^t \Omega(\sigma) \d\sigma}$,
we see that its relative condition number with respect to
$t$ is $|t\om(t)|$, determined by the local frequency,
whereas with respect to
$\om_0$ it is $\int_{t_0}^t \om(\sigma) \d\sigma$, the total phase
accrued over $(t_0,t)$.
The loss in achievable accuracy due to the latter $\kappa$
is well known, and may be interpreted as 
inevitable accuracy loss in taking the sine or cosine of the phase function
\cite{bremer2018}.
This motivates the following.
\begin{defn}[Condition number and achievable accuracy for an oscillatory region]
  \label{kappa}
  Let $u$ be a solution to \cref{ode} on $(t_0,t_1)$ of the form
  $u(t) = e^{z(t)}$ where $\im z'(t) > 0$ for all $t\in(t_0,t_1)$.
We define the condition number of the problem of evaluation of $u(t)$
  by
  \be
  \kappa(t) \;:=\; \max\left[ \, |tz'(t)|, \, \im \, ( z(t) - z(t_0) )\, \right]
  \label{conditionnodef}
  \ee
  and then the resulting best achievable accuracy using finite precision by
  \be
  \kappa \cdot \veps_{\mathrm{mach}}.   \label{mineps}
  \ee
\end{defn}
In practice we use a numerical phase function solution to estimate $\kappa$.
In our examples the first (frequency) term does not dominate the second (accrued phase) term.
In the case where there is no such global phase function
representation $u=e^z$, 
we estimate this second term
by summing the absolute differences in $\im z$ across
the various intervals in each of which a local phase function exists.
Note that this definition does not
capture additional $\kappa$ fluctuations inherent in real-valued
functions:
taking the real part of the toy example, $f_{\om_0}(t) = \cos \om_0 t$,
$\kappa$ with respect to $t$ becomes $|t\om_0 \tan \om_0 t|$,
which fluctuates wildly from $0$ to $\infty$,
although a typical value is again $\bigO(\om_0)$.

In the below tests we denote absolute error in a numerical solution by
$\Delta u(t) := |u(t) - \tilde u(t)|$, where
$u$ is a reference solution computed either analytically
or by a converged reference solver,
and $\tilde u$ is the solution given by the method under test.

All computations were performed on an Intel Xeon Gold 6244,
16-core, 3.6 GHz workstation with 264 GB of RAM. No attempt has been made to
parallelize our code, and for the sake of fair comparison all experiments were
run with a single thread (\texttt{OMP\_NUM\_THREADS=1}).
The open-source software implementing the proposed method and the above
experiments in \texttt{Python} is available on GitHub\footnote{\url{https://github.com/fruzsinaagocs/riccati}}.
The code relies on the \texttt{numpy} \cite{numpy2020} and \texttt{scipy} \cite{scipy2020}
packages to perform linear algebra. For the numerical experiments, the
following versions of software were used: \texttt{Python} 3.8.12,
\texttt{numpy} 1.23.1, \texttt{scipy} 1.9.3.

\subsection{The Airy equation}\label{airy-demo}

A simple but effective test for highly oscillatory solvers is the Airy ODE,
\be
u'' + tu = 0, \quad t \in [1, 10^8],
\label{airy}
\ee
with initial conditions
\be
u(1) = \text{Ai}(-1) + i\text{Bi}(-1), \quad u'(1) = - \text{Ai}'(-1) -i\text{Bi}'(-1).
\label{airyic}
\ee
where Ai, Bi are the Airy functions \cite[Chapter~9]{DLMFairy}.
Its unique solution is
\be
u(t) = \text{Ai}(-t) + i\text{Bi}(-t).
\ee
This is challenging
for standard numerical methods, due to the growing frequency $\om(t) = \sqrt{t}$,
but is expected
to be easy for specialized oscillatory methods since the change in $\om(t)$ per wavelength becomes arbitrarily small, making an asymptotic approximation (such as WKB) 
an increasingly good local approximation.

Choosing a tolerance $\veps=10^{-12}$,
\cref{airy-results} shows the
internal steps our numerical solver takes while solving
\cref{airy}--\cref{airyic},
color-coded by step type, as well as the progression of stepsizes,
the numerical relative error function $\Delta u(t)/u(t)$,
and number of oscillation periods traversed in a single step,
$n_{\text{osc}}$.
The
stepsize $h$ is expected to reflect the timescale over which $\om$ changes,
therefore $h \propto \om(t)/\om'(t) \propto t$, which is observed.
The numerical
error traces the theoretical minimum, $\kappa \cdot \varepsilon_{\text{mach}}$, to
within a digit. The number of oscillations traversed per step can be derived
from $h(t)$: $n_{\text{osc}} \propto \om h \propto t^{3/2}$,
a result clearly supported by the bottom-right figure panel.
Our solver covers around $10^{11}$ periods using $\approx 30$ timesteps,
taking $\approx 10$ milliseconds of CPU time,
while achieving an error around $10^{-11}$ (or, when larger, the theoretical
minimum \cref{mineps}).

\cref{convergence-airy} demonstrates the convergence
of our method on the Airy equation for different solution interval lengths,
achieved by varying $t_1$ with $t_0=1$ fixed.
This shows that the achieved relative accuracy is approximately bounded
(within close to one digit of accuracy)
by the requested tolerance $\varepsilon$,
down to the theoretical minimum error \cref{mineps}.

\begin{figure}[tb]
    \centering
    \includegraphics{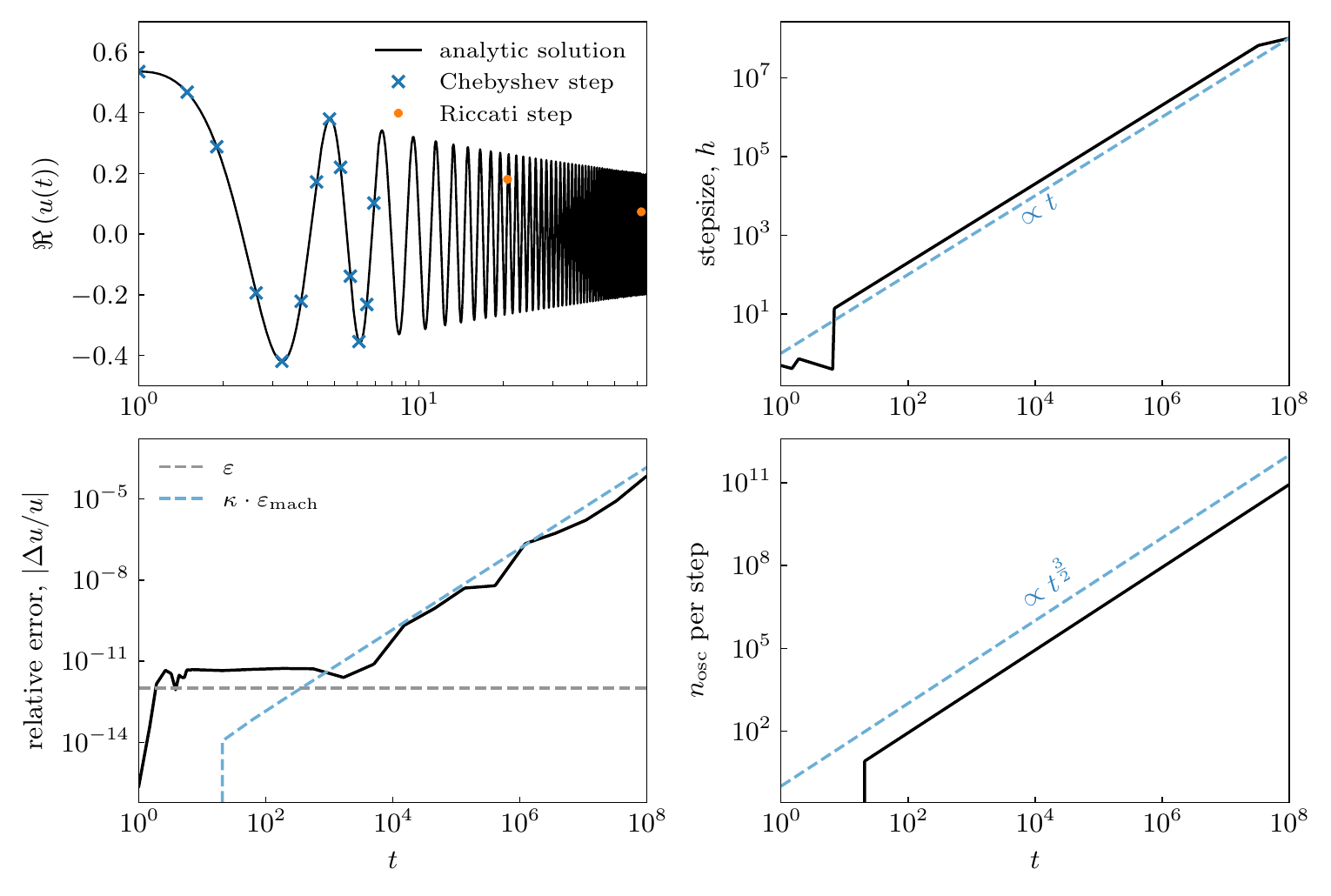}
    \caption{\label{airy-results} Numerical solution of the Airy equation. The
    top left panel shows the real part of the analytic solution in black, on
    top of which the individual timesteps of the solver are plotted, colored
    by their step type. The top right panel plots the stepsize as a function of
    time, which exhibits linear growth. The bottom left shows the relative
    error achieved by the solver in black, with two dashed lines to guide the
    eye: the chosen solver tolerance $\veps$ in gray, and the
    best achievable accuracy \cref{mineps}
    in blue. On the
    bottom right, we show the number of wavelengths of oscillation traversed in
    a single step, as a function of time. 
}
\end{figure}

\begin{figure}[tb]
    \centering
    \includegraphics{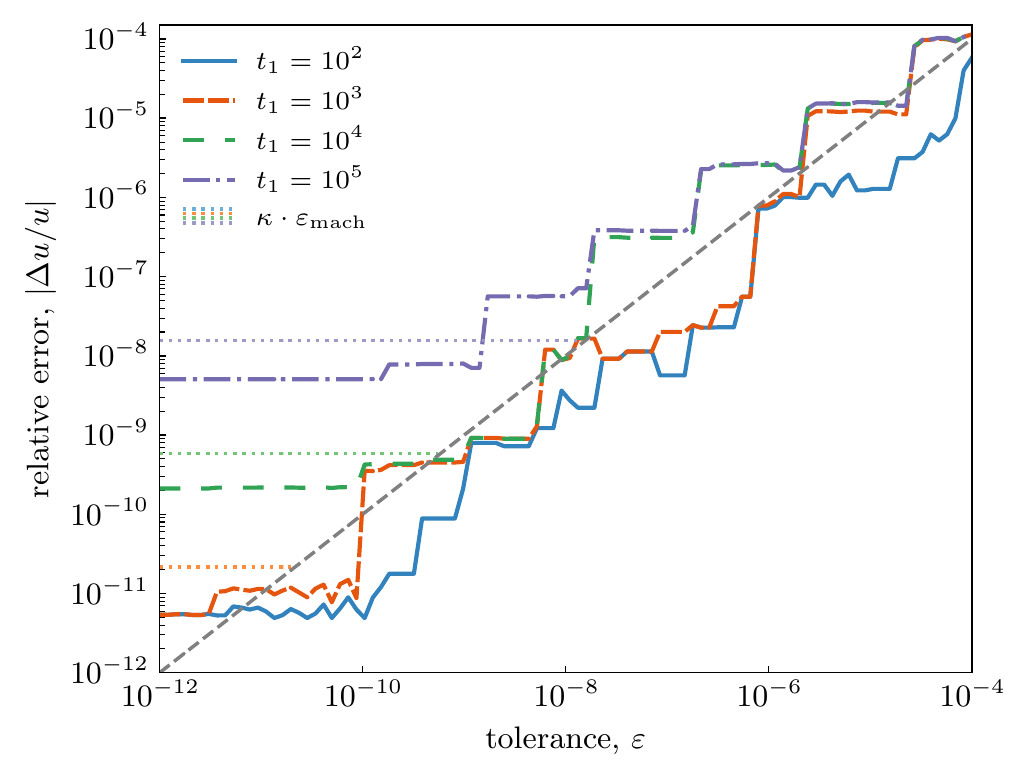}
    \caption{\label{convergence-airy} Error performance of the solver when
    applied to the Airy equation, \cref{airy}--\cref{airyic}. The plot shows
    the achieved relative error of the solver vs the requested tolerance
    for the example in \cref{airy-demo}, while the upper end of the integration
    interval $t_1$ is varied. The dotted lines show the theoretical minimum
    \cref{mineps}
    for each interval, which the solver indeed matches within a small constant.
}
\end{figure}

\subsection{Comparison with standard and state-of-the-art oscillatory solvers \label{solvercomp}}

Here we compare the performance of the present solver, which
we now refer to as adaptive Riccati defect correction (ARDC), to that of the
following recent specialized methods for oscillatory ODEs:
\begin{description}
\item[Kummer's phase function method:]{ an arbitrarily high-order solver
        described by Bremer \cite{bremer2018}, specifically created for highly
        oscillatory problems with neither a friction term nor
        transitions to nonoscillatory behavior. The user provides $\om$ as a function.
We use the \texttt{Fortran 90} implementation provided by its author
        at
        \url{https://github.com/JamesCBremerJr/Phase-functions}
}
\item[oscode:]   {an adaptive solver which switches from low-order WKB (in oscillatory regions) to Runge--Kutta (in smooth regions)
        \cite{agocs2020efficient,agocs2020dense}, that can accommodate a
        friction term. It requires
        only requires the user to specify $\om$ and $\g$, which can be given as
        either closed-form functions or as timeseries. The solver has both a
        \texttt{Python} and \texttt{C++} interface, but for the experiments below it is called in
        \texttt{Python}.}
\item[The WKB marching method] {\cite{arnold2011wkb,korner2022wkb}:
  a solver using
        the mechanism of \cite{agocs2020efficient} to switch between
        WKB-inspired and Runge--Kutta
        methods. Its convergence order in the oscillatory regions is one higher
        than that of \texttt{oscode}, but does not (easily) generalize to
        include a $\g$ term. It requires the user to provide
        the derivative functions of $\om$ up to $\om^{(5)}$.
It is implemented in \texttt{MATLAB}.
    }
\end{description}

As the test problem we take the ODE in \cite[Eq.~(237)]{bremer2018},
\be\label{bremer237eq}
u'' + \om^2(t) u = 0, \qquad \mbox{ where} \;\; \om^2(t) = \lambda^2(1 - t^2\cos3t)
\ee
on $t \in [-1, 1]$, subject to the initial conditions
\be\label{bremer237ic}
u(-1) = 0, \quad u'(-1) = \lambda,
\ee
where we vary the frequency scaling $\lambda$ between $10$ and $10^7$.

For a reference solver we 
use the Kummer's phase function method with $\veps = 10^{-14}$ (its default
tolerance), since it is arbitrarily high
order, and exhibited around $13$ accurate digits
in its published tests at $\lambda = 10$ \cite[Table~1]{bremer2018}.
As expected,
the ill-conditioning of the problem as $\lambda$
grows
reduces its accuracy, giving only $8$ digits by $\lambda = 10^7$.
This is in accordance with the theoretical minimum error \cref{mineps},
thus there will be no meaning given to reported errors below these.
Note that conventional solvers are ruled out as a reference since
the number of oscillations we test are simply too large.

\cref{bremer237tab} presents the runtime statistics of ARDC, Kummer's phase
function method, the WKB marching method, and \texttt{oscode} respectively. 
For all solvers, we set the relative tolerance to $\veps=10^{-12}$, and
in all cases work in double precision. For our solver, we set the number of
Chebyshev nodes in Riccati (oscillatory) steps (see \cref{chebysteps}) to $n = 40$, the
tolerance for stepsize selection to $\varepsilon_h = 10^{-13}$, and the number
of Chebyshev nodes in the spectral (nonoscillatory) steps to the default
$n = 16$. For all other methods, none of the input parameters were modified from their default values
with the exception of the local tolerance.

We now describe the quantities tabulated in \cref{bremer237tab}:
\begin{description}
    \item[$\bm{\mathrm{max}|\Delta u/u|}$:]{Maximum relative error over the
        integration range, evaluated at the timesteps taken internally by the
        solver. If less than the error of the reference Kummer method
    as reported in \cite[Table~1]{bremer2018}, the latter is stated as its upper bound.}
    \item[$\bm{t_{\mathrm{solve}}}$:]{Total runtime, in seconds, of a single
        ODE solve, averaged over between $10^3$ and $10^5$ runs. This is equivalent to the
        ``phase function construction time'' of
        \cite[Table~1]{bremer2018}.}

    \item[$\bm{n_{\mathrm{s,osc}}}$:]{Number of time steps of oscillatory type.
        When this
        quantity appears as $(n_1, n_2)$, $n_1$ denotes the number of attempted
        steps of the given type, out of which $n_2$ have been successful
        (accepted). The same convention applies to $n_{\text{s,slo}}$ and
        $n_{\text{s,tot}}$.}
    \item[$\bm{n_{\mathrm{s,slo}}}$:]{Number of time steps of ``standard''
      (Runge--Kutta or Chebyshev spectral) type, \ie in the nonoscillatory regions of the solution.}
    \item[$\bm{n_{\mathrm{s,tot}}}$:]{Total number of steps performed by the
        solver, sum of $n_{\mathrm{s,osc}}$ and $n_{\mathrm{s,slo}}$.}
    \item[$\bm{n_f}$:]{Total number of function evaluations during a single ODE solve.}
    \item[$\bm{n_{\mathrm{LS}}}$:]{Total number of linear solves
        performed by ARDC.}
\end{description}

\begin{table*}[tb]
    \renewcommand{\arraystretch}{1.2}
    \resizebox{\textwidth}{!}{\begin{tabular}{l c l c c c c c c}
\hline \hline 
 method &  $\lambda$  &  $\max|\Delta u/u|$  &  $t_{\mathrm{solve}}$/\si{\s} &  $n_{\mathrm{s,osc}}$  &  $n_{\mathrm{s,slo}}$  &  $n_{\mathrm{s,tot}}$  &  $n_{\mathrm{f}}$  &  $n_{\mathrm{LS}}$  \\ \hline
ARDC & $10^1$  &  $2.01 \times 10^{-11}$  &  $1.81 \times 10^{-2}$  &  $(24, 0)$  &  $(24, 24)$  &  $(48, 24)$  &  $11388$  &  $55$ \\ 
 & $10^2$  &  $6.3 \times 10^{-13}$  &  $2.42 \times 10^{-3}$  &  $(4, 2)$  &  $(2, 2)$  &  $(6, 4)$  &  $1830$  &  $5$\\ 
 & $10^3$  &  $\leq 3 \times 10^{-12}$  &  $4.70 \times 10^{-4}$  &  $(2, 2)$  &  $(0, 0)$  &  $(2, 2)$  &  $732$  &  $1$\\ 
 & $10^4$  &  $\leq 5 \times 10^{-11}$  &  $4.05 \times 10^{-4}$  &  $(2, 2)$  &  $(0, 0)$  &  $(2, 2)$  &  $732$  &  $1$\\ 
 & $10^5$  &  $\leq 3 \times 10^{-10}$  &  $3.91 \times 10^{-4}$  &  $(2, 2)$  &  $(0, 0)$  &  $(2, 2)$  &  $732$  &  $1$\\ 
 & $10^6$  &  $\leq 5 \times 10^{-9}$  &  $3.79 \times 10^{-4}$  &  $(2, 2)$  &  $(0, 0)$  &  $(2, 2)$  &  $732$  &  $1$\\ 
 & $10^7$  &  $\leq 4 \times 10^{-8}$  &  $3.59 \times 10^{-4}$  &  $(2, 2)$  &  $(0, 0)$  &  $(2, 2)$  &  $732$  &  $1$\\ 
\hline \hline
Kummer's phase function & $10^1$  &  $\leq 7 \times 10^{-14}$  &  $6.74 \times 10^{-3}$  &  $$  &  $$  &  $$  &  $75921$  &  $$ \\ 
 & $10^2$  &  $\leq 5 \times 10^{-13}$  &  $3.76 \times 10^{-3}$  &  $$  &  $$  &  $$  &  $44054$  &  $$\\ 
 & $10^3$  &  $\leq 3 \times 10^{-12}$  &  $2.92 \times 10^{-3}$  &  $$  &  $$  &  $$  &  $24740$  &  $$\\ 
 & $10^4$  &  $\leq 5 \times 10^{-11}$  &  $2.74 \times 10^{-3}$  &  $$  &  $$  &  $$  &  $20521$  &  $$\\ 
 & $10^5$  &  $\leq 3 \times 10^{-10}$  &  $3.19 \times 10^{-3}$  &  $$  &  $$  &  $$  &  $21808$  &  $$\\ 
 & $10^6$  &  $\leq 5 \times 10^{-9}$  &  $3.73 \times 10^{-3}$  &  $$  &  $$  &  $$  &  $23525$  &  $$\\ 
 & $10^7$  &  $\leq 4 \times 10^{-8}$  &  $3.18 \times 10^{-3}$  &  $$  &  $$  &  $$  &  $22468$  &  $$\\ 
\hline \hline
WKB marching & $10^1$  &  $1.95 \times 10^{-11}$  &  $2.39 \times 10^{-1}$  &  $0$  &  $0$  &  $1655$  &  $18293$  &  $$ \\ 
 & $10^2$  &  $2.78 \times 10^{-10}$  &  $3.10 \times 10^{0}$  &  $0$  &  $0$  &  $16181$  &  $179861$  &  $$\\ 
 & $10^3$  &  $\leq 3 \times 10^{-12}$  &  $1.24 \times 10^{0}$  &  $0$  &  $0$  &  $4718$  &  $84601$  &  $$\\ 
 & $10^4$  &  $1.15 \times 10^{-8}$  &  $5.12 \times 10^{-4}$  &  $0$  &  $0$  &  $4$  &  $33$  &  $$\\ 
 & $10^5$  &  $4.35 \times 10^{-8}$  &  $4.79 \times 10^{-4}$  &  $0$  &  $0$  &  $4$  &  $33$  &  $$\\ 
 & $10^6$  &  $1.14 \times 10^{-6}$  &  $4.89 \times 10^{-4}$  &  $0$  &  $0$  &  $4$  &  $33$  &  $$\\ 
 & $10^7$  &  $3.99 \times 10^{-6}$  &  $5.00 \times 10^{-4}$  &  $0$  &  $0$  &  $4$  &  $33$  &  $$\\ 
\hline \hline
\texttt{oscode} & $10^1$  &  $1.41 \times 10^{-12}$  &  $4.88 \times 10^{-1}$  &  $0$  &  $2716$  &  $2716$  &  $89408$  &  $$ \\ 
 & $10^2$  &  $1.44 \times 10^{-11}$  &  $3.79 \times 10^{0}$  &  $0$  &  $27270$  &  $27270$  &  $905938$  &  $$\\ 
 & $10^3$  &  $5.58 \times 10^{-11}$  &  $3.38 \times 10^{1}$  &  $12$  &  $240737$  &  $240749$  &  $7965804$  &  $$\\ 
 & $10^4$  &  $\leq 5 \times 10^{-11}$  &  $1.93 \times 10^{-2}$  &  $137$  &  $0$  &  $137$  &  $4620$  &  $$\\ 
 & $10^5$  &  $\leq 3 \times 10^{-10}$  &  $2.61 \times 10^{-2}$  &  $198$  &  $0$  &  $198$  &  $6226$  &  $$\\ 
 & $10^6$  &  $\leq 5 \times 10^{-9}$  &  $7.13 \times 10^{-2}$  &  $504$  &  $0$  &  $504$  &  $17094$  &  $$\\ 
 & $10^7$  &  $\leq 4 \times 10^{-8}$  &  $7.47 \times 10^{-2}$  &  $465$  &  $0$  &  $465$  &  $18128$  &  $$\\ 
\hline \hline
\end{tabular}
 }
    \caption{\label{bremer237tab} Accuracy, runtime and evaluation statistics of the algorithms
      considered  when applied to \cref{bremer237eq}.
      See \cref{solvercomp} for a summary of the solvers being compared and
      their settings, and an explanation of column headers.}
\end{table*}

\cref{bremer237-timing} provides a visual comparison of the runtime (left) and
achieved relative error (right) as a function of the frequency parameter
$\lambda$ of the solvers being tested. In addition to runtimes and relative
errors at $\veps = 10^{-12}$ reported in \cref{bremer237tab} (solid lines), it
also shows those at $\veps = 10^{-6}$ (dashed lines). It further shows the same
quantities for a 7,8th order Runge--Kutta method (denoted RK78) \cite{HWN,dop853} as a
representative of many standard high-order methods, whose runtime is expected
to grow as $\bigO(\lambda)$. The right-hand panel contains a gray shaded area
whose upper edge
serves as an upper bound on any data points (relative errors) that fall within it. This area is
defined by the error reported for Kummer's phase function method in
\cite[Table~1]{bremer2018}, which being our reference for computing errors
gives an upper bound. This is reflected in \cref{bremer237tab}: for any
measured errors lower than the reference we report only the upper bound.

\cref{bremer237tab} shows clearly the extremely quick convergence of the asymptotic method within ARDC
at sufficiently large frequencies: at $\lambda \geq 10^2$, the number of
function evaluations and steps become constant ($\approx 2$), causing the runtime to be
constant as well. Towards $\lambda = 10^7$, only $\approx 2$ Riccati iterations are required to achieve an estimated (local) relative error of
$10^{-12}$. At $\veps = 10^{-12}$, ARDC achieves the same accuracy as the
Kummer's phase function method, in the intermediate-to-large $\lambda$ regime, but is roughly 10 times faster.
 \cref{bremer237-timing} illustrates this nicely, but also highlights
that Kummer's phase function method becomes significantly faster (comparable 
to ARDC) at $\veps = 10^{-6}$, without the loss of accuracy, at all frequencies.
In other words, the convergence curve (achieved versus requested accuracy) in
the Kummer's phase function implementation is closer to a step function
than to a straight line. 

\texttt{oscode} and the WKB
marching method also achieve fast convergence at sufficiently large
frequencies, but due to the fact that neither are adaptive in the order of the
asymptotic expansion (\ie the number of terms in the expansion) this
happens later, at around $\lambda \geq 10^4$. This manifests itself as the
low-$\lambda$ ``hill'' in the left panel of \cref{bremer237-timing} at $\veps = 10^{-12}$,
but disappears at the higher tolerance $\veps = 10^{-6}$. 
Potentially due to the
stepsize-update algorithm used by \texttt{oscode}, its number of steps,
function evaluations, and therefore runtime, grows slowly rather than staying
constant as the frequency increases. As with Kummer's phase function method, at
sufficiently high frequencies ($\lambda \geq 10^4$), setting a tolerance of
$\varepsilon = 10^{-6}$ in \texttt{oscode} results in the same accuracy as
setting $\varepsilon = 10^{-12}$ but at fewer function evaluations and steps,
again due to the WKB expansion being highly accurate in this regime and to
\texttt{oscode} potentially overestimating its local error.

Finally, the WKB marching method needs significantly fewer function evaluations
(although achieves fewer digits of accuracy) at high frequencies because it
asks the user to provide high-order derivatives of the frequency $\om(t)$,
therefore it need not use further evaluations of $\om$ to compute said
derivatives numerically.

\begin{rmk}\label{dense}
Three solvers (ARDC, the
Kummer's phase function method, and \texttt{oscode}) offer a dense output
option, and do so by interpolating a slowly-varying phase function,
therefore they all have evaluation times of
$\mathcal{O}(10^{-6})$ \si{\s} per new target.
\end{rmk}

\begin{figure}[tb]
    \centering
    \includegraphics{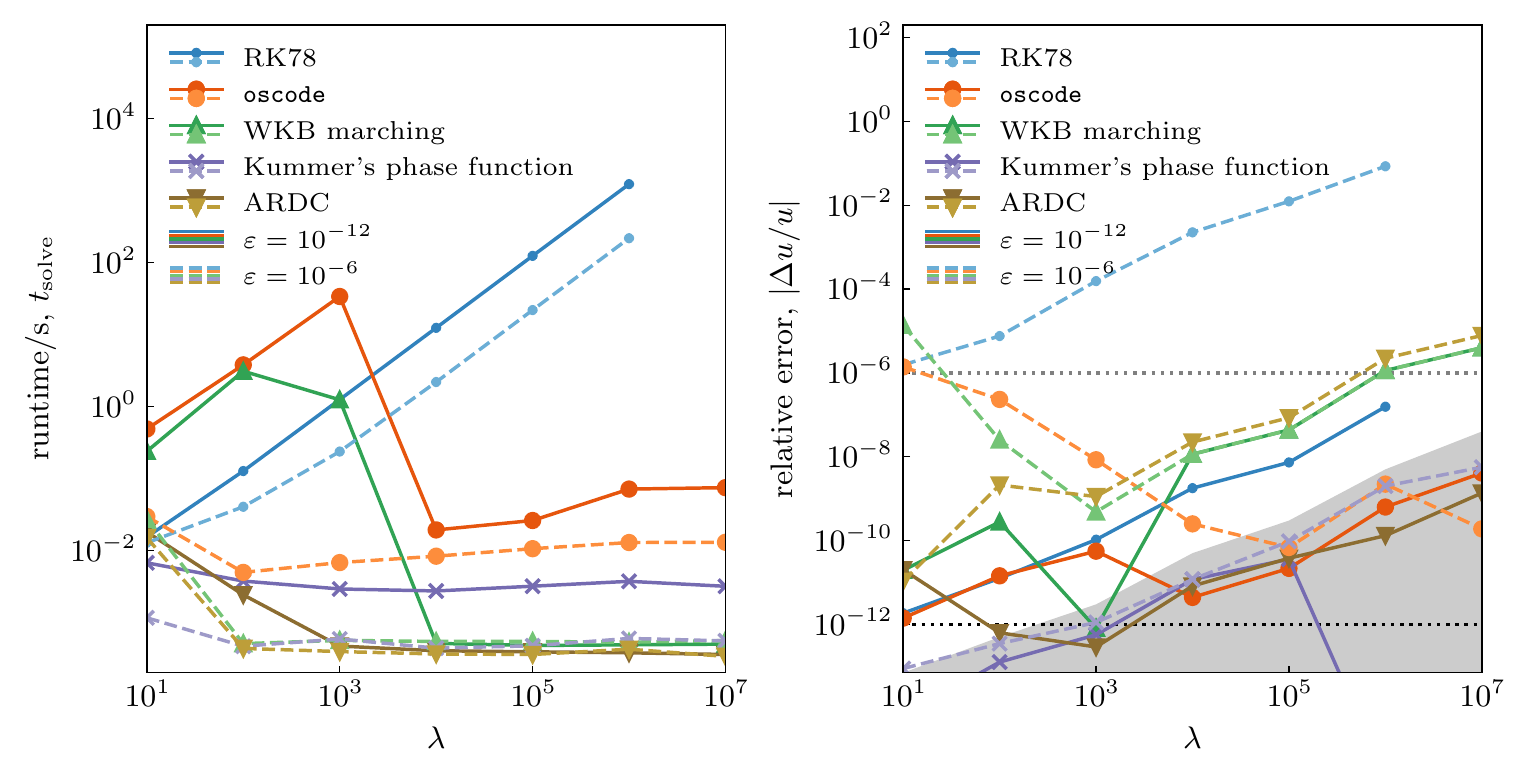}
    \caption{\label{bremer237-timing} Performance comparison of the proposed ARDC method against
    other state-of-the-art methods for \cref{bremer237eq}. The left panel
    shows the total runtime of the solvers listed in \cref{solvercomp} as the
    frequency parameter $\lambda$ is varied, for two different relative
    tolerance settings: runtimes with $\varepsilon = 10^{-12}$ are plotted with
    solid, and $\varepsilon = 10^{-6}$ with dashed lines. The right panel shows
    the corresponding relative error the solvers achieved at the end of the
    integration interval. The gray shaded region in this plot serves as an
    upper bound on any errors that fall within it, since the errors plotted
    here were computed using the Kummer's phase function method as a reference, which is
    only accurate up to the upper edge of the gray region. This is consistent with the theoretical error minimum \cref{mineps}.}
\end{figure}

\subsection{Special function evaluation: Legendre's equation}

Here we apply the proposed ARDC method to the rapid evaluation of
Legendre functions of high degree $\nu$.
Such functions solve Legendre's differential equation
\cite[Sec.~14.2]{dlmf}
\be\label{legendreode}
(1-t^2)u'' - 2tu' + \nu(\nu+1)u = 0, \quad -1 < t < 1.
\ee
We test a range of integer $\nu$ values from $10$ to $10^9$,
and check accuracy against
\texttt{scipy}'s
evaluation of the Legendre polynomial $P_\nu(t)$.
We solve on the solution interval $t \in [0.0, 0.9]$,
with the internal steps of ARDC as test points.\footnote{Since both $\om(t) = \sqrt{\nu(\nu+1)/(1-t^2)}$ and $\gamma(t) = -t/(1-t^2)$ have singularities at $t = \pm 1$,
we defer analysis of this limit to future work.
For now we merely state that empirically we are able to take
$t_1 \to 1$ as $\nu$ grows.}
Initial conditions at $t_0 = 0$ are chosen to match the value and derivative of
$P_\nu$, expressible in terms of gamma functions
\cite[Chapter~14.5]{dlmf} (where the log gamma function is used to avoid
roundoff at large $\nu$).

Error and timing results are shown in \cref{legendre-results}, for tolerances of
$\varepsilon = 10^{-12}$, $\varepsilon_h = 10^{-13}$, and the number of
Chebyshev nodes in Riccati steps set to $n = 16$.
They show a frequency-independent runtime,
show that the required tolerance (up to that allowed by the
condition number
of the problem) is achieved.

Keeping in mind \cref{dense}, ARDC thus enables
(at close to theoretical error bounds)
the evaluation of $P_\nu(t)$ with a startup time of around $10^{-3}$ s,
plus around $10^{-6}$ s per subsequent evaluation at each $t$ argument,
with both costs independent of $\nu$ for large $\nu$.
For comparison, \texttt{scipy}'s evaluation time grows with $\nu$,
taking around 1 s per evaluation at $\nu=10^9$.

\begin{table}
    \centering
    \renewcommand{\arraystretch}{1.2}
    \begin{tabular}{l c c c c c c c}
\hline \hline 
 $\nu$  &  $\max|\Delta u/u|$  &  $t_{\mathrm{solve}}$/\si{\s} &  $n_{\mathrm{s,osc}}$  &  $n_{\mathrm{s,slo}}$  &  $n_{\mathrm{s,tot}}$  &  $n_{\mathrm{f}}$  &  $n_{\mathrm{LS}}$  \\ \hline
$10^1$  &  $1.04 \times 10^{-11}$  &  $6.28 \times 10^{-3}$  &  $(0, 0)$  &  $(12, 12)$  &  $(12, 12)$  &  $3141$  &  $25$\\ 
$10^2$  &  $1.92 \times 10^{-10}$  &  $3.21 \times 10^{-2}$  &  $(35, 3)$  &  $(52, 52)$  &  $(87, 55)$  &  $14345$  &  $107$\\ 
$10^3$  &  $2.63 \times 10^{-12}$  &  $1.75 \times 10^{-3}$  &  $(10, 10)$  &  $(0, 0)$  &  $(10, 10)$  &  $1194$  &  $1$\\ 
$10^4$  &  $5.01 \times 10^{-12}$  &  $1.44 \times 10^{-3}$  &  $(10, 10)$  &  $(0, 0)$  &  $(10, 10)$  &  $1194$  &  $1$\\ 
$10^5$  &  $1.06 \times 10^{-10}$  &  $1.36 \times 10^{-3}$  &  $(10, 10)$  &  $(0, 0)$  &  $(10, 10)$  &  $1194$  &  $1$\\ 
$10^6$  &  $3.83 \times 10^{-10}$  &  $1.31 \times 10^{-3}$  &  $(10, 10)$  &  $(0, 0)$  &  $(10, 10)$  &  $1194$  &  $1$\\ 
$10^7$  &  $1.5 \times 10^{-9}$  &  $1.27 \times 10^{-3}$  &  $(10, 10)$  &  $(0, 0)$  &  $(10, 10)$  &  $1194$  &  $1$\\ 
$10^8$  &  $3.31 \times 10^{-8}$  &  $1.23 \times 10^{-3}$  &  $(10, 10)$  &  $(0, 0)$  &  $(10, 10)$  &  $1194$  &  $1$\\ 
$10^9$  &  $3.85 \times 10^{-7}$  &  $1.24 \times 10^{-3}$  &  $(10, 10)$  &  $(0, 0)$  &  $(10, 10)$  &  $1194$  &  $1$\\ 
\hline \hline 
\end{tabular}
     \caption{Accuracy, runtime and evaluation statistics of the proposed ARDC method when
      applied to Legendre functions via their ODE \cref{legendreode}.
      The degree is $\nu$; other column
      headers are identical to those in \cref{bremer237tab}
and are explained in the text. \label{legendre-results}}
\end{table}

\section{Conclusions \label{conclusions}}

We presented an efficient numerical algorithm for solving linear second order
ODEs of a single variable, posed as initial value problems (see \cref{ode}),
whose solution
may be highly oscillatory for some regions of the solution interval.
It switches adaptively between a new defect correction iteration for
the Riccati phase function (the logarithm of a solution in oscillatory regions), and
a conventional Chebyshev collocation solver (in smooth regions).
Both schemes are discretized to arbitrarily high order, and
have step-size adaptivity.
The resulting run time is independent of the maximum oscillation frequency.
It improves over prior high-frequency solvers
in several ways: it is efficient regardless of
whether the solution oscillates or not; it is arbitrarily high-order accurate;
it solves a more general problem by not imposing the restriction $\gamma =
0$ in \cref{ode}; and it finds a nonoscillatory local phase function
without needing a windowing method.

The proposed Riccati defect correction is asymptotic
with respect to large $\om$, but we believe not in general convergent.
However, we proved (\cref{TR}) that, for analytic coefficients with
$\om(t)$ sufficiently large and $\om'$ and $\g$ sufficiently small in a ball,
that (at the continuous level) it reduces the Riccati residual
geometrically by a factor $\bigO(\om)$ per iterate, \emph{temporarily}
up to a certain number of iterations which also scales as $\bigO(\om)$.
The constants are explicit in this theorem, and
predict rapid residual reduction even when the ball over which $\om$
is roughly constant is only a few oscillation periods across.
We leave it to future work to analyze the discretized iteration (\cref{discr}).

Numerical experiments in four ODEs demonstrate that the adaptivity
achieves close to the requested tolerance, up to theoretical
error bounds imposed by the condition number of the problem.
They also show the solver's high performance, allowing rapid solution times
relative to existing oscillatory solvers at comparable accuracies.

We now discuss some future directions.
Although we focused on initial value problems, much of the
proposed scheme would carry over to the discretization of
possibly-highly-oscillatory
homogeneous two-point boundary value problems.
Due to the linearity
of the ODE, any set of auxiliary conditions can be satisfied as a
post-processing step by solving the ODE with two sets of linearly independent
initial conditions and linearly combining the resulting solutions.
We note that a recent preprint addresses
oscillatory ODEs with a slowly-varying inhomogeneous forcing term \cite{serkh2022}.

Within the scope of this work, we only allowed $\om$ real. When $\om$ is
imaginary, any standard numerical solver will pick up the exponentially growing
solution due to rounding error, which will quickly overpower the evanescent
solution. However, in the case where physical decay conditions are known
(such as quantum eigenvalue problems),
it might be possible to use Riccati phase functions to impose these
conditions in a stable fashion.

Spectral methods are usually applied over the entire solution interval (\eg in
\cite{driscoll2008}) rather than within a time-stepping framework. When applied
locally and over a stepsize that can be changed adaptively, a degeneracy
appears between the number of nodes $n$ and the stepsize $h$, in that $n$ can
be increased and $h$ can be decreased to lower the local error.
Efficiency could be increased by
determining the points at which the algorithm is
supposed to switch between Riccati and Chebyshev steps \emph{a priori}, or at
least identifying large regions of the solution interval where the Chebyshev
spectral method will need to be used, and solving the ODE in these regions with a quasi-linear spectral scheme with large $n$ \cite{tref,fortunato2021}.

On the analysis side, connecting residual error back to the solution error
is a future goal; however, this appears quite technical
in the oscillatory case \cite{heitman2015}.

\section*{Acknowledgments}
We have benefited greatly from discussions with Jim Bremer, Charlie Epstein,
Manas Rachh, and Leslie Greengard.
The Flatiron Institute is a division of the Simons Foundation.

\printbibliography
\end{document}